\documentclass{amsart}
\usepackage[a4paper,width={16cm},left=2.5cm,bottom=3cm, top=3cm]{geometry}

\usepackage{hyperref}
\hypersetup{
    colorlinks=true,
    linkcolor=blue,
    citecolor=blue,
    urlcolor=blue,
}

\makeatletter
\newcommand\org@maketitle{}
\newcommand\@authors{}
\let\org@maketitle\maketitle
\def\maketitle{%
	\let\@authors\authors
	\nxandlist{; }{ and }{; }\@authors
	\hypersetup{
		linktocpage=true,
		pdftitle={\@title},
                pdfauthor={\@authors},
                pdfsubject={\subjclassname. \@subjclass},
		pdfkeywords={\@keywords}
	}%
	\org@maketitle
}
\makeatother

\usepackage{amssymb}
\usepackage{rsfso}
\usepackage{mathtools}
\usepackage{microtype}
\usepackage[alphabetic,msc-links]{amsrefs}
\usepackage{enumitem}
\usepackage{amsfonts}
\usepackage{subfig}
\usepackage{tikz}

\usepackage{cellspace} 
\usepackage{makecell}

\usepackage{pgf}
\usepackage{pgfplots}
\usepackage{mathrsfs}
\usetikzlibrary{arrows}
\usetikzlibrary{decorations.pathreplacing}

\usepackage{doi}
\renewcommand{\PrintDOI}[1]{\doi{#1}}

\numberwithin{equation}{section}

\newtheorem{theorem}{Theorem}[section]
\newtheorem{lemma}[theorem]{Lemma}
\newtheorem{corollary}[theorem]{Corollary}
\newtheorem{proposition}[theorem]{Proposition}
\theoremstyle{definition}
\newtheorem{definition}[theorem]{Definition}
\newtheorem{example}[theorem]{Example}

\newtheorem{remark}[theorem]{Remark}

\newcommand{\Mb}{\mathbb{M}}

\newcommand{\R}{\mathbb{R}}

\newcommand{\eps}{\varepsilon}

\def\XXint#1#2#3{{\setbox0=\hbox{$#1{#2#3}{\int}$}
    \vcenter{\hbox{$#2#3$}}\kern-.5\wd0}}

\newcommand{\beq}{\begin{equation}}
\newcommand{\eeq}{\end{equation}}

\mathchardef\ordinarycolon\mathcode`\:
\mathcode`\:=\string"8000
\begingroup \catcode`\:=\active
  \gdef:{\mathrel{\mathop\ordinarycolon}}
\endgroup

\author{Alessandra De Luca}
\address{Alessandra De Luca\newline\indent
Dipartimento di Matematica e Applicazioni
\newline\indent
Universit\`a degli Studi di Milano-Bicocca
\newline\indent
Via Cozzi 55, 20125, Milano, Italy   }
\email{alessandra.deluca@unimib.it}

\author{Matteo Muratori}
\address{Matteo Muratori\newline\indent
Dipartimento di Matematica 
\newline\indent
Politecnico di Milano
\newline\indent
Piazza Leonardo da
Vinci 32, 20133, Milano, Italy  }
\email{matteo.muratori@polimi.it}

\author{Nicola Soave}
\address{Nicola Soave\newline\indent
Dipartimento di Matematica ``Giuseppe Peano''
\newline\indent
Universit\`a degli Studi di Torino
\newline\indent
Via Carlo Alberto 10, 10123, Torino, Italy   }
\email{nicola.soave@unito.it}



\date{}

\title[]{On the subcritical Lane-Emden equation \\ on Riemannian models with polynomial volume growth}

\begin{document}

\begin{abstract}
    We focus on the problems of existence and non-existence of positive solutions for the Sobolev-subcritical Lane–Emden equation on certain Riemannian manifolds (mainly models) with asymptotically negative curvature, which, from the viewpoint of the volume growth of geodesic balls, can be regarded as intermediate settings between the Euclidean and the hyperbolic spaces.
    
    A number of interesting phenomena arise: the subcritical regime naturally divides into three further ranges, characterized by existence phenomena (slightly subcritical), non-existence phenomena (strongly subcritical), and by a mixed behavior where existence and non-existence strongly depend on additional assumptions on the manifold (intermediate). In the intermediate regime, we further show that the radial homogeneous Dirichlet problem in geodesics balls may admit multiple positive solutions, thereby revealing substantial differences with respect to both the Euclidean and the hyperbolic settings.
\end{abstract}

\maketitle

\section{Introduction}
We consider positive solutions of the \emph{Lane-Emden equation} 
\beq\label{LE}
-\Delta u = u^{q}, \quad u>0 \qquad \text{on $\Mb^n$},
\eeq
where $q>1$ and $\Mb^n$ is a complete, noncompact, $n$-dimensional Riemannian manifold, and $\Delta$ denotes the corresponding Laplace-Beltrami operator. The Lane-Emden equation serves as a fundamental prototype for semilinear elliptic equations, with significant applications both in geometry and physics. Its study has played a crucial role in the development of various techniques in nonlinear analysis.

The Euclidean case is by now well understood: when $\Mb^n \equiv \R^n$ (with the standard Euclidean flat metric) and the problem is \emph{subcritical}, \emph{i.e.}\ $1<q<2^*-1$, Gidas and Spruck proved in \cite{GiSp} that \eqref{LE} has no positive solutions (see also \cite{CheLi} for an alternative proof). We recall that $2^*$ stands for the critical exponent for the Sobolev embeddings, that is $2^*=2n/(n-2)$ for $n \ge 3$ and $2^*=\infty$ for $n=2$. On the other hand, in the critical case $q=2^*-1$ the equation has an $(n+1)$-parameter family of explicit solutions, the celebrated \emph{Aubin-Talenti functions}. These functions arise as extremals of the Sobolev inequality associated with the continuous embedding $D^{1,2}(\R^n) \hookrightarrow L^{2^*}(\R^n)$ and constitute the only positive solutions of the equation, as established by Caffarelli, Gidas, and Spruck \cite{CaGiSp} (see again \cite{CheLi} for alternative proofs). In the supercritical regime $q>2^*-1$, positive radial solutions still exist, although they do not belong to the natural energy space $D^{1,2}(\R^n)$: see \cite{RuSt}, \cite[Chapter I.9]{QS}, and references therein. We recall that $D^{1,2}(\R^n)$ denotes the closure of $C^1_c(\R^n)$ with respect to the norm 
\[
\| u \|_{D^{1,2}(\R^n)} := \left(\int_{\R^n} |\nabla u|^2\,dx \right)^\frac12 = \|\nabla u\|_{L^2(\R^n)} \, .
\]
This definition naturally extends to a Riemannian manifolds.

Regarding equation \eqref{LE} in non-flat geometric settings, we preliminarily observe that the sign of the curvature plays a significant role. In fact, on any complete noncompact Riemannian manifold $\Mb^n$ with \emph{non-negative Ricci curvature}, the situation is reminiscent of the Euclidean case: in the subcritical case \eqref{LE} does not have positive solutions \cite{GiSp} and, if the critical equation on $\Mb^n$ has a positive solution $u$ (satisfying mild additional assumptions, such as $u \in D^{1,2}(\Mb^n$), or $u$ not growing too fast at infinity), then $\Mb^n$ is necessarily isometric to $\R^n$, and $u$ is an Aubin-Talenti function. Namely, we have full rigidity for both the manifold and the solution. We refer to the recent paper \cite{CaMo} for the precise statements. See also \cite{BaKr, CiFaPo, FoMaMa} for related results.

In contrast, the problem reveals new features on negatively-curved manifolds. For the sake of simplicity, in this discussion we focus in particular on \emph{Cartan-Hadamard manifolds}, which are complete, noncompact, and simply connected Riemannian manifolds with nonpositive sectional curvatures at every point. A prototypical example is the hyperbolic space $\mathbb{H}^n$, where sectional curvatures are negative and constant. It is well known that any Cartan-Hadamard manifold $ \Mb^n $ is diffeomorphic to $\R^n$, and the exponential map centered at any point is a global diffeomorphism  (see Subsection \ref{bkg} ahead). This allows one to introduce \emph{polar coordinates} centered at any point, and makes the search for \emph{radial solutions} to \eqref{LE} especially meaningful. However, we point out that most of the results we will present are valid in a more general setting, and the Cartan-Hadamard structure will be used rarely.

Among the first significant results in the negatively curved setting, Mancini and Sandeep \cite{MaSa} proved that the subcritical Lane-Emden equation on the hyperbolic space admits solutions. Specifically, there exists a unique positive radial solution in $H^1(\mathbb{H}^n)$. Furthermore, there exist infinitely-many positive radial solutions with infinite energy, see \cite[Theorem 2.3]{BoGaGrVa}. The existence of a finite-energy solution is strictly connected with the validity of a \emph{Poincar\'e inequality} in the whole space $\mathbb{H}^n$, that is, the presence of a \emph{spectral gap} for the Laplacian spectrum. This contrasts sharply with the Euclidean case, where it is well known that \( \lambda_1(-\Delta,\mathbb{R}^n) = 0 \), and reflects the exponential volume growth of balls in \(\mathbb{H}^n\) as opposed to the polynomial growth in \(\mathbb{R}^n\). A generalization of such results to a broader class of negatively-curved manifolds was provided in \cite[Theorem 2.5]{BeFeGr}: roughly speaking, if \(\Mb^n\) is a \emph{model manifold} (not necessarily of Cartan-Hadamard type) with \(\lambda_1(-\Delta,\Mb^n) > 0\), then the subcritical Lane-Emden equation on \(\Mb^n\) admits a positive (radial) solution in $H^1(\Mb^n)$. The precise definition of a model manifold will be recalled in the forthcoming Subsection \ref{bkg}.

With regards to the critical and supercritical regimes, it turns out that the Lane-Emden equation admits infinitely-many positive radial solutions on any Cartan-Hadamard model manifold \cite{BeFeGr, MuSo}. However, if a positive radial solution $u$ exists with \(\nabla u \in L^2(\Mb^n)\), then \(\Mb^n\) must be isometric to \(\mathbb{R}^n\), $q=2^*-1$, and \(u\) must be an Aubin-Talenti function \cite{MuSo}, which restores a form of rigidity as in the non-negatively curved case. A challenging open problem consists in proving (or disproving) that, for the critical and supercritical cases, any positive solution is necessarily radial. This is known only under additional assumptions, such as the validity of the \emph{Cartan-Hadamard conjecture} (currently settled for \(n \leq 4\)), and when \(u\) is also an extremal of the Sobolev inequality, see \cite[Theorem 1.1]{MuSo}. Successful techniques from the Euclidean case, such as the moving planes method, are generally inapplicable to manifolds, except under restrictive geometric conditions, as explored in \cite{AlDaGe}. Moreover, recent integral estimates developed in \cite{CaMo, CiFaPo, FoMaMa} do not seem to be suitable for handling negative curvatures (at least without strong restrictions). We refer to \cite[Remark 1.6 (vi)]{CiFaPo} for a rigidity result on manifolds with possibly nonpositive curvature somewhere; note, however, that such an example does not cover Cartan-Hadamard manifolds and appears to be very peculiar.

In the present paper, we focus on the subcritical case and address a question that, to our knowledge, remains open in its full generality: what happens when \(\Mb^n\) has nonpositive curvature and \(\lambda_1(-\Delta,\Mb^n) = 0\)? Such a geometric setting can be seen, to some extent, as an interpolation between the Euclidean and the hyperbolic cases, and, because of the lack of a spectral gap, it may be tempting to expect a Euclidean-type scenario (for instance, this was suggested in \cite{BeFeGr}). On the contrary, we will demonstrate that entirely new phenomena emerge, including the appearance of two additional critical exponents beyond the classical Sobolev one, which turn out to be very sensitive to the asymptotic volume growth of geodesic balls (assumed to be polynomial -- see the comments below Remark \ref{rem: reg M}).

\subsection{Geometric background and main results}\label{bkg}

In order to state the main results of the paper, we first need to set some notation, recall basic definitions, and introduce our geometric framework. 

Given a complete and noncompact Riemannian manifold $ \Mb^n $, we let $ dV $ denote its Riemannian volume measure and $ \mathrm{d}(\cdot,\cdot) $ the geodesic distance on $ \Mb^n $. For some fixed $ o \in \Mb^n $, which will often play the role of a \emph{pole}, and $R>0$, we set
$$
B_R(o) := \left\{ x \in \Mb^n : \ \mathrm{d}(x,o) < R \right\} ,
$$
that is, the (open) geodesic ball of radius $R$ centered at $o$. When no ambiguity occurs, we will simply write $ B_R $. Also, we let $ \partial B_R(o) $ (or $ \partial B_R $) denote the boundary of $ B_R(o) $ or, equivalently, the set of points at distance $ R $ from $o$. The function 
\beq\label{polrad}
r(x) := \mathrm{d}(x,o) \qquad \forall x \in \Mb^n \setminus \{o\}
\eeq 
is always $1$-Lipschitz regular. Moreover, it is smooth in the set $ \Mb^n \setminus \mathrm{cut}(o) \setminus \{o\} $, where $ \mathrm{cut}(o) $ is the \emph{cut locus} of $ o $, namely the set of points $ x \in \Mb^n \setminus \{ o \} $ at which the geodesic connecting $ o $ to $ x $ ceases to be uniquely minimizing; remarkably, thanks to a well-known result, we have that $ V\!\left( \mathrm{cut}(o) \right) = 0 $ (for more details see \emph{e.g.}~\cite[Subsection 1.2.2]{MRS} and references therein). 

Upon denoting by $ T_o \Mb^n $ the \emph{tangent space} of $ \Mb^n $ at $o$, we recall that for every $ \mathsf{v} \in T_o \Mb^n $ the \emph{exponential map} $ \exp_o (\mathsf{v}) $ is defined as $ \gamma_\mathsf{v}(1) $, where $\gamma_\mathsf{v}$ is the unique geodesic such that $ \gamma_\mathsf{v}(0)=o $ and $ \gamma_\mathsf{v}'(0)=\mathsf{v} $. The exponential map is always well defined locally, that is, for vectors $\mathsf{v}$ with a small norm, and it is actually a \emph{local diffeomorphism}. In general, however, it fails to be a global diffeomorphism: this is precisely related to the fact that the cut locus of $ o $ may not be empty. Nevertheless, there are some important classes of manifolds for which it is indeed a global diffeomorphism; if that happens for a specific $ o \in \Mb^n $, then $ \Mb^n $ is called a \emph{manifold with a pole}, the pole being precisely $ o $ (we refer to the celebrated monograph \cite{GrWu}). On any such a manifold, we can therefore introduce the (global) \emph{polar coordinates} about the pole by defining, for every $x\in \Mb^n\setminus \{o\}$, the corresponding polar radius $r \equiv r(x) $ as the geodesic distance of $x$ from $o$ (recall \eqref{polrad}), and the polar angle $\theta \equiv \theta(x)\in \mathbb{S}^{n-1}$ representing the starting direction, in the tangent space $T_o \Mb^n$, of the (unique) geodesic, emanating from $o$, that realizes such a distance. In particular, we can always write the metric $ \mathsf{g} $ of $ \Mb^n $ as 
\beq \label{metr-pole}
\mathsf{g} \equiv dr^2 + \mathsf{A}_r \, ,
\eeq
where $ dr $ stands for the radial component of a vector in $ T_x \Mb^n $, namely the one in the direction of $\nabla r(x)$, and $ \mathsf{A}_r $ is a suitable metric on the unit sphere $ \mathbb{S}^{n-1} $ or, equivalently, on the geodesic sphere $ \partial B_r $. Moreover, in this special coordinate frame, the \emph{Laplace-Beltrami} operator (Laplacian for short) reads 
\begin{equation}\label{L-B}
    \Delta = \frac{\partial^2}{\partial r^2}+ m(r,\theta) \, \frac{\partial}{\partial r}+ \Delta_{\partial B_r} \, ,
\end{equation}
where $\Delta_{\partial B_r}$ is the Laplace-Beltrami operator on  $\partial B_r$ and $m(r,\theta)$ is a smooth function on $(0,+\infty)\times \mathbb{S} ^{n-1}$ that can explicitly be related to $A_r$, and in fact coincides with $ \Delta r $. 

Regarding curvatures, we will only need to work with \emph{radial} Ricci and sectional curvatures. Specifically, the \emph{radial Ricci curvature} at $ x \in \Mb^n \setminus \{ o \} $, denoted by $ \mathrm{Ric}_o(x) $, is the Ricci curvature at $x$ evaluated in the radial direction $ \nabla r(x) $. Similarly, the \emph{radial sectional curvature}, denoted by $ \mathrm{Sec}_{\omega}(x) $, is the sectional curvature of $ \Mb^n $ at $ x \in \Mb^n \setminus \{ o \} $ corresponding to any $2$-plane $\omega \subset T_x \Mb^n $ that contains the radial direction. In particular, $ \mathrm{Ric}_o(x) $ is the sum of $ n-1 $ radial sectional curvatures associated with mutually orthogonal $2$-planes. Note that, even if $ \Mb^n $ is not a manifold with a pole, radial curvatures are still well defined \emph{outside} of the cut locus.

A remarkable class of Riemannian manifolds with a pole is formed by \emph{Cartan-Hadamard manifolds}, \emph{i.e.}, complete and simply connected Riemannian manifolds with everywhere nonpositive sectional curvatures. Another important class of manifolds with a pole comprises (noncompact)  \emph{Riemannian models} (or \emph{model manifolds}), which are at the core of the present paper. Prior to recalling what a model manifold is, it is convenient to introduce the following set of real functions. 

\begin{definition}\label{def:mf}
   We say that $ \psi : [0,+\infty) \to [0,+\infty) $  is a \emph{model function} if $\psi \in C^\infty\!\left( [0,+\infty) \right) $, $ \psi>0 $ in $ (0,+\infty) $,
   \begin{equation}\label{hp 1 psi}
\psi(0)=0 \qquad \text{and} \qquad\psi'(0)=1 \, ,
\end{equation}
 and 
\beq\label{hp 1 psi'}\psi^{(2k)}(0)=0 \qquad \forall k \in \mathbb{N} \, .
\eeq   
\end{definition}
In most of the paper we restrict to the case of (noncompact) \emph{model manifolds}, that is, special manifolds with a pole $o \in \Mb^n$ such that the metric is given, in polar (or spherical) global coordinates about $o$, by
\[
\mathsf{g}=dr^2 + \psi^2(r) \, \mathsf{g}_{\mathbb{S}^{n-1}} \, ,
\]
for some model function $\psi$ complying with Definition \ref{def:mf}, where $ \mathsf{g}_{\mathbb{S}^{n-1}}$ stands for the usual round metric on the unit sphere $\mathbb{S}^{n-1}$ (see \emph{e.g.}\ \cite[Chapter 2]{GrWu}). In other words, the metric $ \mathsf{A}_r $ in \eqref{metr-pole} is conformal to the round metric on $ \mathbb{S}^{n-1} $ for every $ r>0 $.  It is worth mentioning that the choice $\psi(r)=r$ corresponds to the Euclidean space $\mathbb{R}^n$, while $\psi(r)=\sinh r$ is the model function associated with the hyperbolic space $\mathbb{H}^n$. Also, the radial sectional and Ricci curvatures read
\beq\label{sec-rad}
\mathrm{Sec}_\omega (x)  = -\frac{\psi''(r)}{\psi(r)} \qquad \forall x \equiv (r,\theta) \setminus \{ o \} 
\eeq
and 
$$
 \mathrm{Ric}_o(x) = -(n-1) \, \frac{\psi''(r)}{\psi(r)} \qquad \forall x \equiv (r,\theta) \setminus \{ o \} \, ,
$$
respectively. In particular, from \eqref{sec-rad} (and the expression for non-radial curvatures -- see \cite[Subsection 2.2]{GMV} and references therein) we can infer that a noncompact model manifold is Cartan-Hadamard if and only if $ \psi $ is \emph{convex}. Moreover, the function $ m(r,\theta) $ in \eqref{L-B} is purely radial and explicitly related to $\psi$ by the formula
$$
m(r,\theta)=(n-1) \, \frac{\psi'(r)}{\psi(r)} \, ,
$$
whereas
$$
\Delta_{S_r} = \frac{1}{\psi^2(r)} \, \Delta_{\mathbb{S}^{n-1}} \, .
$$
Hence, in the special case when $ f \equiv f(r) $ is a purely \emph{radial function}, it holds
\begin{equation}\label{lap-rad}
  \Delta f  = f'' + (n-1) \, \frac{\psi'}{\psi} \, f' \, .
\end{equation}
We will often deal with radial functions: whenever no ambiguity occurs, we will interpret them both as functions of $ x \in \Mb^n$ via $r(x)$ and as one-variable functions of $ r \in [0,+\infty) $.

\begin{remark}\label{rem: reg M}
Most of the analysis we will carry out actually works under the weaker assumptions that $ \psi \in C^1([0,+\infty)) \cap  C^\infty ((0,+\infty)) $ and satisfies \eqref{hp 1 psi}. Such conditions ensure that the corresponding model manifold $\Mb^n$ is $C^1$ near $ o $. More in general, $ \Mb^n $ is $C^m$ near $ o $ if $\psi \in C^m([0,+\infty))$ and \eqref{hp 1 psi'} (along with \eqref{hp 1 psi}) holds for all $ k \in \mathbb{N} $ with $ 2k \le m  $, namely derivatives of even order up to $m$ must vanish at $r=0$.
\end{remark}

We have already anticipated that in \cite{BeFeGr} the authors considered the Lane-Emden equation on certain model manifolds, and proved existence of an $H^1$ positive solution under the condition $\lambda_1\!\left(-\Delta, \Mb^n\right)>0$. The precise set of hypotheses made in \cite{BeFeGr} is the following: in addition to \eqref{hp 1 psi}, the authors required \eqref{hp 1 psi'} for $k=1$, 
\[
\psi'>0 \quad \text{in $[0,+\infty)$} \qquad \text{and} \qquad \liminf_{r \to +\infty} \frac{\psi'(r)}{\psi(r)} \in (0,+\infty] \, .
\]
Such assumptions allow them to show that the first Laplacian eigenvalue on $\Mb^n$ is indeed positive \cite[Lemma 4.1]{BeFeGr}, which, together with the compactness of the embedding $H^1_{\mathrm{rad}}(\Mb^n) \hookrightarrow L^{p}(\Mb^n)$ for every $2<p<2^*$, ensures existence of (radial) solutions to \eqref{LE} for every $ q \in (1,2^*-1)$ via classical variational methods.

In what follows, we focus instead on cases when 
\beq\label{hp 2 psi}
\lim_{r \to +\infty} \frac{\psi'(r)}{\psi(r)}=0 \, ,
\eeq
which, in particular, implies that $\lambda_1\!\left(-\Delta,\Mb^n\right)=0$ (see again \cite[Lemma 4.1]{BeFeGr}). As we will see, under \eqref{hp 2 psi}, the existence or non-existence of solutions to \eqref{LE} is very sensitive to the rate at which $ \psi $ grows at infinity. For that reason, we will often assume that at least one of the following extra conditions holds:
\beq\label{hp 00 psi}
\exists \, \alpha >1 \, , \ \kappa_1,\kappa_2 >0: \quad \kappa_1 \, r^{\alpha} \le \psi(r) \le \kappa_2 \, r^{\alpha} \qquad \text{for all $r$ large enough} \, ,
\eeq
\beq\label{hp 0 psi}
\exists \, \alpha >1 \, , \ \kappa >0: \quad \psi(r) \sim \kappa \, r^{\alpha} \qquad \text{as $r \to +\infty$} \, ,
\eeq
\beq\label{hp 3 psi}
\exists \, \alpha >1 \, , \ \kappa>0: \quad \psi'(r) \sim \kappa \alpha  \,  r^{\alpha-1} \qquad \text{as $r \to +\infty$} \, ,
\eeq
where $f(r) \sim g(r)$ means that $f(r)/g(r) \to 1$ as $r \to +\infty$. Clearly, the above conditions are increasingly more restrictive, and  under \eqref{hp 3 psi} we have that \eqref{hp 2 psi} holds with a specific rate:
\begin{equation}\label{psi'psi}
  \frac{\psi'(r)}{\psi(r)} \sim \frac{\alpha}{r} \rightarrow 0 \qquad \text{as $r \to +\infty$} \, .  
\end{equation}
Moreover, any of them entails that the volume of geodesics balls centered at the pole $o$ (or at any other point) satisfies
\[
V(B_R) = O\big(R^{\alpha(n-1)+1}\big) \qquad \text{as $R \to +\infty$} \, ,
\]
that is, an intermediate growth between the Euclidean one $O(R^n)$ and the hyperbolic one $O\big(e^{(n-1)R}\big)$.

Somehow surprisingly, it is indeed necessary to prescribe such a behavior, since for any value of $\alpha>1$ we can identify three different regimes for the existence or non-existence of solutions to \eqref{LE}, which strongly depend on $\alpha$. To this end, it is convenient to introduce the following \emph{new critical exponents}, in addition to $2^*$:
\[
\tilde 2_\alpha := \frac{\alpha(n-1)+1}{\alpha(n-1)-1} \qquad \text{and} \qquad
2_\alpha^* := \frac{2(\alpha(n-1)+1)}{\alpha(n-1)-1} = 2\cdot \tilde 2_\alpha \, .
\]
Here and throughout the rest of the paper, we will implicitly assume that $n \ge 2$ (note that both $ \tilde 2_\alpha $ and $ 2^*_\alpha $ are finite numbers).

\smallskip

We are now in a position to state all of our main results. For the reader's convenience, we split them into three paragraphs, according to the following denomination of the regimes depending on $q$:
$$
\begin{cases}
q \in \left( 1 , \tilde 2_\alpha \right] & \Rightarrow \ \  \text{\it strongly subcritical} \, , \\
q \in \left(2^*_\alpha-1 , 2^*-1 \right) & \Rightarrow \ \  \text{\it sightly subcritical} \, , \\ 
q \in \left( \tilde 2_\alpha , 2^*_\alpha-1 \right] & \Rightarrow \ \ \text{\it intermediate} \, .
\end{cases}
$$
Note that the borderline case $q=2^*_\alpha-1$ has been formally assigned to the intermediate regime, but at some points it will be convenient to treat it as an element of the slightly subcritical regime as well.

\subsection*{The \emph{strongly subcritical} regime} If $q$ is less than or equal to $ \tilde 2_\alpha $, then not only \eqref{LE} does not have positive radial solutions, but even positive supersolutions, \emph{not necessarily radial}, do not exist.

\begin{theorem}\label{thm: supersol}
Let $ \alpha>1 $. Let $ \Mb^n $ be a model manifold associated with a model function $\psi$ satisfying
\begin{equation}\label{bound-lap 0}
\frac{\psi'(r)}{\psi(r)} \le \frac{K}{r}  \qquad \forall r>0
\end{equation}
and
\begin{equation}\label{bound-psi}
\psi(r) \le K \, r^\alpha \qquad \forall r \ge r_0 \, ,
\end{equation}
for some $K,r_0>0$. Then, for every $q \in \left(1,\tilde 2_\alpha \right]$, the Lane-Emden equation \eqref{LE} does not have nonnegative nontrivial supersolutions; namely, if $u \ge 0$ satisfies
\begin{equation}\label{distrib}
-\Delta u \ge u^q \qquad \text{on $\Mb^n$}
\end{equation}
in the sense of distributions, then $u \equiv 0$.
\end{theorem}

We stress again that in the above statement $u$ \emph{need not be radial}, although the geometric setting is. This non-existence result has a natural counterpart in the Euclidean setting, where it is known that the Lane-Emden equation has no nonnegative nontrivial supersolutions for every $q \in \left(1,\tilde 2\right]$, where $\tilde 2=n/(n-2)$ for $n \ge 3$ and $\tilde 2=\infty$, with $ q \in (1,\infty) $, for $n=2$ (note that $\tilde 2=\tilde 2_1$). For more details, we refer to \cite[Theorem 8.4]{QS} (see also Remark 8.5 therein for further related results). 

Also, we also point out that the range of exponents $q$ for which non-existence of supersolutions is ensured is sharp, see Theorems \ref{soprasol-modelli} and \ref{prop: str count} below. This already justifies the fact that the sole vanishing condition \eqref{hp 2 psi} would not suffice in order to obtain results of this type: indeed, the exponent $\tilde 2_\alpha$ depends on the choice of $\alpha$, whereas \eqref{hp 2 psi} does not. Specifically, if assumption \eqref{hp 3 psi} is in force, then both \eqref{bound-lap 0} and \eqref{bound-psi} hold, therefore different choices of $\alpha$ give different (sharp) ranges for existence or non-existence of nonnegative nontrivial supersolutions. 

The proof of Theorem \ref{thm: supersol} is inspired by the Euclidean case. In fact, the strategy (suitably refined) even works in more general geometric contexts (not necessarily Riemannian models), as the next result shows.

\begin{theorem}\label{thm: non ex str gen}
Let $ \alpha>1 $ and $ \Mb^n $ be a complete noncompact Riemannian manifold. Assume that, for some $ o \in \Mb^n $, either there exist $ Q>0 $ and $ r_o>0 $ such that
\begin{equation}\label{ricci-bound-1}
    \mathrm{Ric}_o (x) \ge - (n-1) \, \frac{Q}{r^2} \quad \forall x \in B_{r_o}^c(o) \setminus \mathrm{cut}(o)  \qquad \text{and} \qquad \limsup_{R \to +\infty} \frac{V\!\left(B_R(o)\right)}{R^{\alpha(n-1)+1}} < +\infty \, ,
\end{equation}
or there exists $ r_o >0 $ such that 
\begin{equation}\label{ricci-bound-2}
 \mathrm{Ric}_o (x) \ge - (n-1) \, \frac{\alpha(\alpha-1)}{r^2} \qquad \forall x \in B_{r_o}^c(o) \setminus \mathrm{cut}(o) \, .
\end{equation}
Then, for every $q \in \left(1,\tilde 2_\alpha\right]$, the Lane-Emden equation \eqref{LE} does not have nonnegative nontrivial supersolutions.
\end{theorem}

As already anticipated, we can construct counterexamples to the previous statements for $q>\tilde 2_\alpha$. 


\begin{theorem}\label{soprasol-modelli}
Let $ \alpha>1 $. Let $ \Mb^n $ be a model manifold associated with a model function $\psi$ satisfying
\begin{equation}\label{bound-lap}
\psi' \ge 0 \qquad \text{and} \qquad \liminf_{r \to +\infty} \frac{\psi'(r)}{\psi(r)} \, r  \ge \alpha \, .
\end{equation}
Then, for every $q > \tilde 2_\alpha$, the Lane-Emden equation \eqref{LE} has a positive, radial, and smooth supersolution.
\end{theorem}


\begin{theorem}\label{prop: str count} Let $\alpha>1$ and $\Mb^n$ be a Cartan-Hadamard manifold. Assume that, for some $ o \in \Mb^n $, for every $\delta \in (0 , \alpha-1) $ there exists $r_\delta>0$ such that 
\begin{equation}\label{curvaturasezpiccola}
  \mathrm{Sect}_{\omega}(x) \le - \frac{(\alpha-\delta)(\alpha-\delta -1)}{r^2} \qquad \forall x \in \Mb^n \setminus B_{r_\delta}(o) \, .
\end{equation}
Then, for every $q > \tilde 2_\alpha$, the Lane-Emden equation \eqref{LE} has a positive, radial, and smooth supersolution.
\end{theorem}


Regarding Theorem \ref{soprasol-modelli}, it is clear that the result is valid for any nondecreasing model function $\psi$ satisfying \eqref{hp 3 psi}.

\subsection*{The \emph{slightly subcritical} regime} We now focus on the case when $q$ is ``close'' to $2^*-1$ (or $q$ is sufficiently large for $n=2$), that is, $ q \in \left[2^*_\alpha-1,2^*-1\right) $. In this regime we can prove the existence, under mild assumptions on $\psi$, of a radial \emph{Sobolev minimizer}, which is in particular a positive \emph{finite-energy} solution.

\begin{theorem}\label{da capire}
Let $ \alpha>1 $. Let $ \Mb^n $ be a model manifold associated with a model function $\psi$ satisfying
\begin{equation}\label{eq-MR}
  \frac{\psi'(r)}{\psi(r)} \ge \frac{\sigma'(r)}{\sigma(r)} \qquad \forall r > 0 \, ,
\end{equation}
where $ \sigma $ is another model function complying with condition \eqref{hp 00 psi}.
Then, for every $q \in \left(2^*_\alpha-1,2^*-1\right)$, the Lane-Emden equation \eqref{LE} has a positive radial solution in $D^{1,2}_{\mathrm{rad}}(\Mb^n)$, which is also a Sobolev minimizer. 
\end{theorem}

In the statement, we denoted by $D^{1,2}_{\mathrm{rad}}(\Mb^n)$ the closure of the space of $ C^1_c $ radial functions $f$ with respect to the weighted norm
\[
 \left(\int_{\Mb^n} \left| \nabla f \right|^2 dV \right)^{\frac 1 2} = \left(\int_0^{+\infty} \left| f' \right|^2 \psi^{n-1} \, dr \right)^{\frac 1 2} .
\]
By a \emph{Sobolev minimizer}, we mean a function $ f \in D^{1,2}_{\mathrm{rad}}(\Mb^n)$ achieving
\beq\label{SobolevFunct}
I_{\psi,q}:=\inf_{f \in D^{1,2}_{\mathrm{rad}}(\Mb^n) \setminus \{0\}} R_{\psi,q}(f) \, , \qquad \text{where} \quad R_{\psi,q}(f) := \frac{\left(\int_0^{+\infty} \left|f'\right|^2 \psi^{n-1} \, dr \right)^\frac{1}{2}}{\left(\int_0^{+\infty} |f|^{q+1} \, \psi^{n-1}\,dr \right)^\frac{1}{q+1}} \, .
\eeq
In general, $D^{1,2}_{\mathrm{rad}}(\Mb^n)$ is not a function space, as in the closure operation constants may pop up and therefore need to be removed by means of a quotient. This occurs, for instance, if $ n=2 $ and $\alpha=1$. However, if a Sobolev-type inequality holds, which will always be the case for us since $\alpha>1$, it is indeed a function space with norm $ \| \nabla f \|_2 $.
In fact, in proving Theorem \ref{da capire}, crucial tools are given by the sharp one-dimensional inequalities provided by \cite[Theorems 6.2 and 7.4]{OK}, which will allow us to establish the continuity of the embedding $D^{1,2}_{\mathrm{rad}}(\Mb^n) \hookrightarrow L^p(\Mb^n)$ for every $p \in \left[2^*_\alpha, 2^*\right]$, along with its compactness for $ p \in \left(2^*_\alpha , 2^* \right)$, under \eqref{hp 00 psi} or, more in general, under assumption \eqref{eq-MR}. For the reader's convenience, the corresponding statements are recalled in Section \ref{sec: slight}, see Proposition \ref{KO} and Corollary \ref{KO-cor} there. In this way, we can use a direct variational argument to prove existence of a solution to \eqref{LE} as a Sobolev minimizer. As mentioned above, it is hopeless to be able to obtain such a kind of results assuming only \eqref{hp 2 psi}, since the specific asymptotic behavior of $ \psi$ (in terms of $\alpha$) plays a major role here.

The borderline case $ q=2^*_\alpha-1 $ is much subtler, due to lack of compactness. We can still obtain existence of Sobolev minimizers, but under more restrictive assumptions on $\psi$ (this is however consistent with the non-existence results we will present in the next paragraph).

\begin{theorem}\label{exi-crit}
    Let $\alpha>1$. Let $ \Mb^n $ be a model manifold associated with a model function $\psi$ satisfying \eqref{hp 0 psi}. Suppose that the following inequality holds between optimal Sobolev quotients:
\beq\label{asym-sobo}
I_{\psi,q} < I_\alpha \qquad \text{for } q=2^*_\alpha-1 \, , 
\eeq
where
\[
\quad I_\alpha:= \inf_{\varphi \in C^1_c([0,+\infty)) \setminus \{ 0 \} } \frac{\left(\int_0^{+\infty} \left| \varphi'(r) \right|^2 \kappa^{n-1} \, r^{\alpha(n-1)} \, dr \right)^{\frac 12}}{\left(\int_0^{+\infty} \left|\varphi(r)\right|^{2^*_\alpha} \kappa^{n-1} \,r^{\alpha(n-1)} \, dr \right)^{\frac{1}{2^*_\alpha}}} \, .
\]
    Then the Lane-Emden equation \eqref{LE} has a positive radial solution in $D^{1,2}_{\mathrm{rad}}(\Mb^n)$ for $q=2^\ast_\alpha-1$, which is also a Sobolev minimizer.
\end{theorem}

\begin{remark}
If $\psi$ satisfies \eqref{hp 00 psi}, then the assumptions of Theorem \ref{da capire} are clearly met. On the other hand, the theorem is more flexible than that, since one can easily construct examples of model functions $\psi$ for which \eqref{hp 00 psi} does not hold (with $\psi$ growing faster than $r^\alpha$), while \eqref{eq-MR} is still satisfied. In this regard, we refer to Subsection \ref{sub: examples exi-crit} below for some explicit examples.
Therein, we also provide examples of functions $\psi$ fulfilling \eqref{hp 0 psi} and \eqref{asym-sobo}.
\end{remark}

\subsection*{The  \emph{intermediate} regime} It remains to discuss what happens when $q \in \left(\tilde 2_\alpha, 2^*_\alpha-1\right]$. This case turns out to be particularly delicate. We firstly present a non-existence result.

\begin{theorem}\label{thm: non-ex int}
Let $\alpha>1$. Let $\Mb^n$ be a model manifold associated with a model function $\psi$ satisfying \eqref{hp 00 psi}. Let $q \in \left(\tilde 2_\alpha, 2^*_\alpha-1 \right]$, where in the critical case $ q=2^*_\alpha-1 $ we additionally require \eqref{hp 3 psi} to hold. Suppose moreover that
\beq\label{hp 4 psi}
(n-1) \, \frac{\psi'(r)}{\psi^{n}(r)} \int_0^r \psi^{n-1} \,ds \le \frac12+\frac{1}{q+1} \qquad \forall r>0 \, .
\eeq
Then the Lane-Emden equation \eqref{LE} has no positive radial solution.
\end{theorem}

\begin{remark}\label{rem: on hp 4}
Regarding assumption \eqref{hp 4 psi} we observe that, for $q \le 2^*_\alpha-1$, the strict inequality is always valid for any model function $\psi$ for \emph{small} $r>0$ (by \eqref{hp 1 psi}). Indeed, 
\[
\frac{\psi'(r)}{\psi^{n}(r)} \int_0^r \psi^{n-1} \,ds \to \frac{1}{n} \qquad \text{as $r \to 0$} \, ,
\]
hence the validity of the strict inequality in \eqref{hp 4 psi} reduces to 
\[
\frac12+\frac1{q+1}-\frac{n-1}{n}>0 \qquad  \iff \qquad q+1<2^* \, ,
\]
which is plainly satisfied for all $q \le 2^*_\alpha-1$.
\end{remark}

\begin{remark}
If $q < 2^*_\alpha - 1$ and \eqref{hp 3 psi} is in force in the place of \eqref{hp 00 psi}, then \eqref{hp 4 psi} is also fulfilled for sufficiently large values of $r > 0$ (with strict inequality). Indeed, in this case
\[
\frac{\psi'(r)}{\psi^{n}(r)} \int_0^r \psi^{n-1} \,ds \rightarrow \frac{\alpha}{\alpha(n-1)+1} \qquad \text{as $r \to +\infty$} \, ,
\]
and
\[
\frac12+\frac1{q+1} -\frac{\alpha(n-1)}{\alpha(n-1)+1}>0 \qquad \iff \qquad q+1<2^*_\alpha \, .
\]
Roughly speaking, this means that, at least under the more restrictive condition \eqref{hp 3 psi}, assumption \eqref{hp 4 psi} is not automatically granted only in a certain interval of radii $[r_0,r_1]$, with $0<r_0<r_1<+\infty$ (depending on $\psi$).
\end{remark}

We point out that assumption \eqref{hp 4 psi} is used in order to ensure the monotonicity of a \emph{Pohozaev-type function}, which in turn is the key ingredient in the non-existence result (see the beginning of Subsection \ref{nonex} below). The connection between monotonicity properties of a Pohozaev-type function and the structure of positive radial solutions of semilinear problems was explored, in different contexts than ours, in \cite{BeFeGr, BoGaGrVa, GaItQu, MuSo, MuSoSys}. 

We will provide several examples of model functions $\psi$ satisfying \eqref{hp 3 psi} and \eqref{hp 4 psi} in Subsection \ref{sub: examples non-ex} ahead. It is natural to wonder whether an assumption such as \eqref{hp 4 psi} is necessary or not. This is, in general, an open problem, but we can show that some further restrictions on $\psi$ (beyond \eqref{hp 00 psi} or \eqref{hp 3 psi}) are indeed necessary. That is, we have the following existence result.

\begin{theorem}\label{thm: ex int}
Let $\alpha>1$ and $q \in \left(\tilde 2_\alpha, 2^*_\alpha-1\right)$. Then there exists a Cartan-Hadamard model manifold $\Mb^n$ associated with a model function $\psi$, satisfying \eqref{hp 3 psi}, such that the Lane-Emden equation \eqref{LE} has a positive radial solution $ u \in D^{1,2}_{\mathrm{rad}}\!\left( \Mb^n \right) $.
\end{theorem}

Theorems \ref{thm: non-ex int} and \ref{thm: ex int} reveal that in the (strict) intermediate regime $q \in \left(\tilde 2_\alpha, 2^*_\alpha-1\right)$ existence and non-existence depend in a subtle way on the specific structure of $\psi$, and that knowing only its asymptotic behavior as $r \to +\infty$ is not enough. The same phenomenon also occurs for the critical case $q=2^*_\alpha-1$ (cf.\ Theorems \ref{thm: non-ex int} and \ref{exi-crit}). We emphasize that in Theorem \ref{thm: ex int} we construct an example of a $C^\infty$ \emph{Cartan-Hadamard} model manifold: thus, existence or non-existence is not even a matter of regularity or convexity. This marks a stark difference with the rigidity results, on Cartan-Hadamard models, obtained in \cite{MuSo, MuSoSys} for the \emph{supercritical} case.

Interestingly, the existence of a global positive radial solution in the intermediate regime implies a \emph{non-uniqueness} phenomenon for local homogeneous Dirichlet problems in balls, which is in sharp contrast with the Euclidean and hyperbolic cases, see \cite[Theorem 1.2]{NiNu}, \cite[Proposition 4.4]{MaSa}, and \cite[Lemma 4.3]{BeFeGr} for an extension to more general model manifolds.

\begin{theorem}\label{non-uniq-loc}
Let $\alpha>1$. Let $ \Mb^n $ be a model manifold associated with a model function $\psi$ satisfying \eqref{hp 00 psi}. Suppose that the Lane-Emden equation \eqref{LE} has a positive radial solution for some $q \in \left( \tilde{2}_\alpha , 2^*_\alpha-1 \right) $. Then there exists some $R>0$ such that the Dirichlet problem
\begin{equation}\label{loc-dir}
\begin{cases}
-\Delta u = u^q \, , \ u>0  & \text{in } B_R \, , \\ 
u = 0  & \text{on } \partial B_R \, ,
\end{cases}
\end{equation}
admits at least two radial solutions.
\end{theorem}

Note that, by combining Theorems \ref{thm: ex int} and \ref{non-uniq-loc}, we can assert that there exist Cartan-Hadamard model manifolds for which \eqref{loc-dir} admits multiple solutions for some $ q \in (1,2^*-1) $ and $R>0$.

\begin{remark}
    In fact, we could have replaced the assumption that $\psi$ satisfies \eqref{hp 00 psi} with the weaker requirement 
    \[
    \frac{\psi'(r)}{\psi(r)} \le \frac{\sigma'(r)}{\sigma(r)} \qquad \forall r >0\,,
    \]
    for some model function $\sigma$ complying with \eqref{hp 00 psi} (compare with Theorem \ref{da capire}). In the statement, we decided to ask directly \eqref{hp 00 psi} for the sake of simplicity.
\end{remark}

\begin{remark}
    Reviewing the proof of Theorem \ref{non-uniq-loc} (see Subsection \ref{sub:nonuniq} ahead), one realizes that the thesis is even stronger: that is, there exists some $ R>0 $ such that problem \eqref{loc-dir} admits two different radial solutions, which are in addition \emph{Sobolev minimizers} of the quotient 
    $$
    \inf_{ f \in H^1_{0,\mathrm{rad}}(B_R) \setminus \{ 0 \} } \frac{\left(\int_0^{R} \left| f' \right|^2 \psi^{n-1} \, dr \right)^{\frac 12}}{\left(\int_0^{R} \left|f\right|^{q+1} \psi^{n-1} \, dr \right)^{\frac{1}{q+1}}} \, .
    $$
\end{remark}

\begin{remark}
    An interesting open problem concerns the structure of the set of all positive (radial) solutions in the cases where existence is known. On the hyperbolic space, we have already recalled that this was settled in \cite{MaSa} and \cite{BoGaGrVa}. In that framework, a key tool seems to be the \emph{uniqueness} of positive finite-energy radial solutions, which is in turn strictly related to the uniqueness of positive radial solutions of the \emph{homogeneous Dirichlet problem} posed in geodesic balls. However, for general $\psi$ such a local uniqueness may fail, as shown in Theorem \ref{non-uniq-loc}. This result, although concerned with a special range of exponents, suggests that the classification of solutions for a general $\psi$ is subtle and challenging.
\end{remark}

To conclude, for the sake of clarity, we present a summary table of the results obtained in the different regimes, considering a generic model manifold $ \Mb^n $ whose model function $\psi$ satisfies \eqref{hp 3 psi} for some $\alpha>1$: 

 \setlength{\tabcolsep}{0.1cm}
\setlength{\cellspacetoplimit}{4pt}
\setlength{\cellspacebottomlimit}{4pt}
\begin{table}[h!]
\centering
\begin{tabular}{|Sc|Sc|Sc|Sc|Sc|}
\hline
Range & \makecell{Existence of nonnegative \\ nontrivial supersolutions}  & \makecell{Existence of positive \\ radial solutions}  & \makecell{Uniqueness of positive \\ radial solutions in balls}  \\ 
\hline 
 $q \in \left(1,\tilde 2_\alpha\right]$      & \it NO & \it NO & \it unknown \\
 $q \in \left(2^*_\alpha-1, 2^*-1\right)$
     & \it YES & \it YES  & \it unknown \\
     $ q=2^*_\alpha-1 $  & \it YES &  \it both YES and NO depending $\psi$ & \it unknown  \\ 
$q \in \left(\tilde 2_\alpha, 2^*_\alpha-1\right)$  & \it YES & \it both YES and NO depending $\psi$ & \it it fails for some $\psi$ \\
\hline
\end{tabular}
\end{table}

\subsection*{Structure of the paper} In Section \ref{sec: strong sub} we focus on the strongly subcritical regime, proving Theorems \ref{thm: supersol}, \ref{thm: non ex str gen}, \ref{soprasol-modelli}, and \ref{prop: str count}. Section \ref{sec: slight} is concerned with the slightly subcritical regime, proving Theorems \ref{da capire} and \ref{exi-crit}. Finally, in Section \ref{sec: int} we prove all of the results regarding the intermediate regime, that is, Theorems \ref{thm: non-ex int}, \ref{thm: ex int}, and \ref{non-uniq-loc}. Examples of model manifolds fulfilling the assumptions of the main theorems in the slightly subcritical and intermediate regimes are provided in Subsections \ref{sub: examples exi-crit} and \ref{sub: examples non-ex}, respectively.


\section{The {strongly subcritical} regime}\label{sec: strong sub}

In the first part of the section we prove the non-existence results, that is, Theorems \ref{thm: supersol} and \ref{thm: non ex str gen}. In the second part, we show the existence of nonnegative nontrivial supersolutions for $q>\tilde 2_\alpha$, namely Theorems \ref{soprasol-modelli} and \ref{prop: str count}.

\subsection{Non-existence of nonnegative nontrivial supersolutions for $q \le \tilde 2_\alpha$}
Although Theorem \ref{thm: supersol} can be deduced from Theorem \ref{thm: non ex str gen}, we provide here a short and direct proof, presented for the sake of clarity. The method is a generalization of the one employed in \cite{QS}, in the proof of Theorem 8.4 therein, within the Riemannian setting.

\begin{proof}[Proof of Theorem \ref{thm: supersol}]
Let $\varphi\in C^\infty_c([0,+\infty))$ be such that 
$$
0\leq \varphi\leq 1\, , \qquad  \varphi'\leq 0 \, , \qquad \varphi\equiv 1 \quad \text{in } [0,1] \, , \qquad \varphi\equiv 0 \quad \text{in } [2,+\infty) \, .
$$
For any $R>r_0$, we introduce the cutoff function 
\begin{equation}\label{v-cutoff}
\phi_R(x) := \varphi^m\!\left(\frac{r(x)}{R}\right) \qquad \forall x \in \Mb^n \, , 
\end{equation}
for some $ m \ge 2 $ to be determined. As is customary, in the sequel, when no ambiguity occurs we will also interpret $\phi_R$ as a purely real function of the variable $r$, without relabeling it. Hence, by explicit computations, recalling \eqref{lap-rad} and using \eqref{bound-lap 0} in combination with the decreasing monotonicity of $\varphi$, we have that for every $r\in [R,2R]$
\begin{equation}\label{phi-R-deriv}
\begin{aligned}
\Delta \phi_R(r) =\,& \phi_R''(r)+(n-1) \, \frac{\psi'(r)}{\psi(r)} \, \phi_R'(r) \\
\geq &\, \frac{m}{R^2} \, \varphi^{m-1}\!\left(\tfrac{r}{R}\right)\varphi''\!\left(\tfrac{r}{R}\right) + \frac{m(m-1)}{R^{2}} \, \varphi^{m-2}\left(\tfrac{r}{R}\right) \left[\varphi'\!\left(\tfrac{r}{R}\right)\right]^2 + m(n-1) \, \frac{K}{rR} \, \varphi^{m-1}\!\left(\tfrac{r}{R}\right)\varphi'\!\left(\tfrac{r}{R}\right).
\end{aligned}
\end{equation}
Taking advantage of the fact that $ \varphi \le 1 $ and $ \varphi',\varphi'' $ are supported in $ [1,2] $, we can deduce from \eqref{phi-R-deriv} that 
\begin{equation*}
- \Delta \phi_R(r) \leq \frac{C}{R^2} \, \varphi
^{m-2} \!\left(\tfrac{r}{R}\right) \chi_{[R,2R]}(r) \, ,
\end{equation*}
for some general positive constant $C>0$ independent of $R$ (which may change below but will not be renamed), where we let $\chi$ denote the characteristic function of a set. Choosing $m= {2q}/{(q-1)}$ in order to have $\varphi^{m-2} = \varphi^{\frac m q}$, we then get 
\begin{equation*}
- \Delta \phi_R(x) \leq \frac{C}{R^2} \, \phi_R^{\frac 1 q}(x) \, \chi_{A_{R,2R}}(x) \, , 
\end{equation*}
where $A_{R,2R}$ stands for the annulus
\begin{equation*}
A_{R,2R}:= \left\{x\in \mathbb{M}^n: \ r(x) \in [R,2R] \right\}.
\end{equation*}
Thanks to this estimate, using $\phi_R$ as a test function in \eqref{distrib} we obtain 
\begin{equation}\label{int-q}
\int _{\mathbb{M}^n} u^q \, \phi_R \, dV \leq -\int _{\mathbb{M}^n} u \, \Delta \phi_R \, dV \leq \frac{C}{R^2} \, \int_{A_{R,2R}} u \, \phi_R^{\frac 1 q} \, dV \, .
\end{equation}
By H\"{o}lder's inequality applied to the rightmost term, and \eqref{bound-psi}, we infer that
\begin{equation}\label{significativa}
\begin{aligned}
\int _{\mathbb{M}^n}u^q \, \phi_R\,dV & \leq \frac{C}{R^2} \left(\int_R^{2R}\psi^{n-1}\,dr\right)^{\frac{q-1}{q}}\left(\int_{A_{R,2R}} u^q \, \phi_R \, dV\right)^{\frac{1}{q}}\\
&\leq C R^{[\alpha(n-1)+1]\frac{q-1}{q}-2} \left(\int_{A_{R,2R}} u^q \, \phi_R\, dV\right)^{\frac{1}{q}} . 
\end{aligned}
\end{equation}
for all $R>0$ sufficiently large. Therefore, it follows that 
\begin{equation}\label{uq-L1}
\int _{\mathbb{M}^n} u^q \, \phi_R \, dV\leq C R^{\alpha(n-1)+1-\frac{2q}{q-1}} \, .
\end{equation}
If $q<\tilde{2}_\alpha$ it is readily seen that $ \alpha(n-1)+1-{2q}/{(q-1)}<0$, so that taking the limit as $R\to + \infty$ gives $u\equiv 0$. If instead $q=\tilde{2}_\alpha$, still from \eqref{uq-L1} we find that $u\in L^q(\mathbb{M}^n)$. Then, going back to \eqref{significativa}, we get
\begin{equation*}
\int _{\mathbb{M}^n} u^q \, \phi_R \, dV \leq C \left(\int_{A_{R,2R}} u^q \, \phi_R\, dV\right)^{\frac{1}{q}} \rightarrow 0 \qquad \text{as $R\to +\infty$} \, .
\end{equation*}
Hence, we can deduce that $u\equiv 0$ in this case as well.
\end{proof}

We now focus on the proof of Theorem \ref{thm: non ex str gen}. To this end, we first recall a classical Laplacian-comparison result, which is key to our strategy.

\begin{proposition}[Theorems 1.11 and 1.13 in \cite{MRS}]\label{LC}
    Let $ \Mb^n $ be a complete noncompact Riemannian manifold, and let $o \in \Mb^n$. Suppose that 
    \[
      \mathrm{Ric}_o(x) \ge -(n-1) \, \frac{\psi''(r)}{\psi(r)} \qquad \forall x \in \Mb^n \setminus \mathrm{cut}(o) \setminus \{o\}
      \]
     for some model function $ \psi \in C^2([0,+\infty)) $ such that $ \psi''(r)/\psi(r) $ can be continuously extended to $r=0$. Then
        \[
      \Delta r \le (n-1) \, \frac{\psi'(r)}{\psi(r)} \qquad \text{in } \Mb^n \, ,
      \]
    in the sense of distributions (or, equivalently, Radon measures), and 
            \[
      V(B_R(o)) \le \omega_{n-1} \, \int_0^R \psi^{n-1} \, dr  \qquad \forall R>0 \, ,
      \]
    where $\omega_{n-1}$ is the Riemannian volume of the $(n-1)$-dimensional unit sphere.
\end{proposition}

In the next lemma, we assume for simplicity that $ B_\mathsf{R}(o) $ is a smooth domain for all $\mathsf{R}>0$. Clearly, this is not always the case, but there is no loss of generality to our purposes (similar comments to \cite[Proof of Theorem 2.1]{BS} apply). Incidentally, it is always true if $\Mb^n$ is a Cartan-Hadamard manifold.

\begin{lemma}\label{lemma-approx-loc}
   Let $ q>1 $ and let $ \Mb^n $ be a complete noncompact Riemannian manifold. Suppose that $ u \ge 0 $ satisfies \eqref{distrib} in the sense of distributions. Let $o \in \Mb^n$. Then, for every $\mathsf{R}>0$, there exist sequences of nonnegative functions $ u_k , \mu_k \in C^\infty(B_\mathsf{R}(o)) $ such that 
   \[
       -\Delta u_k = \mu_k \qquad \text{in } B_\mathsf{R}(o)  \, ,
       \]
      \begin{equation}\label{approx-k-2}
   \lim_{k \to \infty} \int_{B_\mathsf{R}(o)} (-\Delta u_k) \, \varphi \, dV \ge \int_{B_\mathsf{R}(o)} u^q \, \varphi \, dV   \qquad \forall \varphi \in C_c(B_\mathsf{R}(o)): \ \varphi \ge 0 \, ,
   \end{equation} 
   and
   \begin{equation}\label{approx-k-3}
   \lim_{k \to \infty} u_k = u    \qquad \text{in } L^1_{\mathrm{loc}}(B_\mathsf{R}(o)) \,.
   \end{equation}
\end{lemma}
\begin{proof}
First of all, we notice that the supersolution inequality \eqref{distrib}, in combination with the non-negativity of $u$, implies that $ -\Delta u $ is a nonnegative distribution on $\Mb^n$, therefore by standard facts in distribution theory it is actually a nonnegative Radon measure, which we call $ \mu $. Hence, by the local regularity of Radon measures, for every $ \mathsf{R}>0 $ we can pick a sequence of nonnegative functions $ \mu_k \in C^\infty_c(B_\mathsf{R}(o)) $ such that 
      \begin{equation}\label{approx-proof}
   \lim_{k \to \infty} \int_{B_\mathsf{R}(o)}  \varphi \, \mu_k \, dV = \int_{B_\mathsf{R}(o)} \varphi \, d\mu   \qquad \forall \varphi \in C_c(B_\mathsf{R}(o)) \, .
   \end{equation} 
Thanks to critical elliptic regularity, see for instance \cite[Theorem 2 in Section III]{BG} (note that this result is concerned with the solution $u_0$ of the homogeneous Dirichlet problem, but since the $u-u_0 $ is harmonic then $u$ inherits the same regularity as $u_0$), we have that $ u \in W^{1,p}(B_\mathsf{R}(o)) $ for all $ 1 \le p< n/(n-2) $. In particular, the trace of $ u $ on $  \partial B_\mathsf{R}(o)$ is always well defined as a nonnegative $ L^p(\partial B_\mathsf{R}(o)) $ function (at least). Let us therefore consider, for every $  k \in \mathbb{N} $, the following elliptic Dirichlet problem:
   \[
   \begin{cases}
       -\Delta u_k = \mu_k & \text{in } B_\mathsf{R}(o)  \, , \\
       u_k = u \vert_{\partial B_\mathsf{R}(o)} & \text{on } \partial B_\mathsf{R}(o) \, .
       \end{cases}
       \]
Interior elliptic estimates ensure that $ u_k $ is smooth in $B_\mathsf{R}(o)$, and the maximum principle guarantees that $ u_k \ge 0 $. Moreover, still from the critical elliptic theory, there exists a constant $ C_{r,\mathsf{R}}>0 $ (independent of $k$) such that 
$$
\left\| u_k \right\|_{W^{1,p}(B_r(o))} \le C_{r,\mathsf{R}} \left( \left\| \mu_k \right\|_{L^1(B_\mathsf{R}(o))} + 1 \right)
$$
for all $ 0<r<\mathsf{R} $. Since the right-hand side is bounded in $k$, we have that $ \{ u_k \} $ converges strongly (up to subsequences) in $ L^p_{\mathrm{loc}}(B_\mathsf{R}(o)) $ to some function $v\ge0$, which satisfies  
   \[
   \begin{cases}
       -\Delta v = \mu & \text{in } B_\mathsf{R}(o)  \, , \\
       v = u \vert_{\partial B_\mathsf{R}(o)} & \text{on } \partial B_\mathsf{R}(o) \, ,
       \end{cases}
       \]
at least in the very weak sense. On the other hand, such a problem admits $u$ as its unique solution, thus $v=u$ and \eqref{approx-k-3} is proved. As for \eqref{approx-k-2}, it is a direct consequence of \eqref{approx-proof}, the definition of $u_k$, and the fact that $ \mu \ge u^q \, V  $ by construction. 
\end{proof}

\begin{proof}[Proof of Theorem \ref{thm: non ex str gen}]
First of all, let us show that the curvature assumptions \eqref{ricci-bound-1} and \eqref{ricci-bound-2} can be written in a more convenient way that will allow us to appeal to Proposition \ref{LC}. To this end, similarly to \cite[Lemma 4.2]{GMV}, consider the following (continuous) function:
$$
G(r) := 
\begin{cases}
  K  & \text{if } r \in [0,r_o] \, ,\\
  \frac{Q}{4r_o^3} \, (r-r_o) + \frac{K}{r_o} \, (2r_o-r) & \text{if } r \in (r_o,2r_o) \, , \\
  \frac{Q}{r^2}  & \text{if } r \ge 2r_o \, ,
\end{cases}
$$
where $Q$ is as in \eqref{ricci-bound-1} and we define 
$$ 
K := \max\left\{- \tfrac{1}{n-1} \min_{x \in B_{r_o}(o) \setminus \mathrm{cut}\{ o \} } \mathrm{Ric}_o(x) \, , \, \frac{Q}{r_o^2} \right\} .
$$
Then, it is readily seen that with these choices it holds
$$
\mathrm{Ric}_o (x) \ge - (n-1) \, \frac{\psi''(r)}{\psi(r)} \qquad \forall x \in \Mb^n \setminus \mathrm{cut}(o) \setminus \{ o \} \, ,
$$
where $ \psi $ is the unique solution of the ODE problem
\begin{equation}\label{ODEg}
\begin{cases}
   \psi''(r) = G(r) \, \psi(r) & \text{for } r>0 \, , \\
   \psi'(0) = 1 \, , &  \\
   \psi(0) = 0 \, . & 
\end{cases}
\end{equation}
Because $G \ge 0$ is a purely reciprocal quadratic function for $ r \ge 2r_o $, and the differential equation in \eqref{ODEg} is linear, as observed in \cite[Subsection 8.1]{GMV}, we can assert that there exist two constants $ a_1>0 $ and $ a_2 \in \R $ such that   
\begin{equation}\label{psi-approx}
   \psi(r) = a_1 \, r^{\frac{1+\sqrt{1+4Q}}{2}} + a_2 \, r^{\frac{1-\sqrt{1+4Q}}{2}}  \qquad \forall r \ge 2r_o \, .
\end{equation}
As a result, it is plain that 
\[
\frac{\psi'(r)}{\psi(r)} \le \frac{c_0}{r} \qquad \forall  r \ge 2 r_o
\]
and
\[
\int_0^r \psi^{n-1} \, ds \le c_1 \, r^{(n-1)\frac{1+\sqrt{1+4Q}}{2}+1} \qquad \forall r \ge 2r_o \, ,
\]
for suitable constants $c_0,c_1>0$. We are therefore in a position to apply Proposition \ref{LC}, which entails 
\begin{equation}\label{LC-1}
  \Delta r  \le \frac{c_0}{r} \qquad \text{in } B^c_{2r_o}(o)
\end{equation}
and
\begin{equation}\label{LC-2}
 V(B_R(o))  \le c_2 \, R^{(n-1)\frac{1+\sqrt{1+4Q}}{2}+1} \qquad \forall R \ge 2r_o \, ,
\end{equation}
for another suitable constant $c_2>0$. We have thus shown that under the leftmost bound in \eqref{ricci-bound-1} estimates \eqref{LC-1}--\eqref{LC-2} hold; since \eqref{ricci-bound-2} is just \eqref{ricci-bound-1} with $Q=\alpha(\alpha-1)$, in this case \eqref{LC-2} reads 
\begin{equation}\label{LC-2-bis}
  V(B_R(o))  \le c_2 \, R^{\alpha(n-1)+1} \qquad \forall R \ge 2r_o \, ,
\end{equation}
which corresponds to the rightmost bound in \eqref{ricci-bound-1}. In summary, we can and will from now on assume that both \eqref{LC-1} and \eqref{LC-2-bis} are satisfied. 

Next, let us go back to the proof of Theorem \ref{thm: supersol}, and consider the smooth radial cutoff function $ \phi_R $ defined as in \eqref{v-cutoff}. The same computation, written in the current more general framework, yields 
\begin{equation}\label{lap-vR}
\Delta \phi_R(r) = \frac{m}{R} \, \varphi'\!\left( \tfrac r R \right) \varphi^{m-1}\!\left( \tfrac r R \right) \Delta r + \frac{m}{R^2} \left[ \varphi'' \varphi + (m-1) (\varphi')^2 \right]\!\left( \tfrac r R \right) \varphi^{m-2}\!\left( \tfrac r R \right) \left| \nabla r \right|^2  .
\end{equation}
Note that the above expression is consistent, since $ m \ge 2 $, $ \Delta r $ is a Radon measure, and $  x \mapsto r(x) $ is Lipschitz with $ |\nabla r(x)| =1 $ almost everywhere. Since $ \varphi $ is smooth and $ \varphi' \le 0 $, also recalling the cutoff properties of $ \phi_R $, from \eqref{LC-1} and \eqref{lap-vR} it follows that
\begin{equation}\label{lap-vR-bis}
- \Delta \phi_R  \le \frac{C}{R^2} \, \phi_R^{\frac{m-2}{m}} \chi_{B_{2R}(o) \setminus B_R(o)} \, ,
\end{equation}
still in the sense of measures, for some $C>0$ independent of $R>2 r_o$. Now, we test \eqref{lap-vR-bis} with the sequence $ \{ u_k \} $ provided by Lemma \ref{lemma-approx-loc} (for some fixed $ \mathsf{R}>2R $), obtaining 
\[
- \int_{B_{\mathsf{R}}(o)}  \Delta u_k \, \phi_R  \, dV \le \frac{C}{R^2} \, \int_{B_{2R}(o) \setminus B_R(o)} u_k \, \phi_R^{\frac{m-2}{m}} \, dV \, ;
\]
passing to the limit as $ k \to \infty $, taking advantage of \eqref{approx-k-2}--\eqref{approx-k-3}, we end up with 
\[
\int_{\Mb^n}  u^q \, \phi_R  \, dV \le \frac{C}{R^2} \, \int_{B_{2R}(o) \setminus B_R(o)} u \, \phi_R^{\frac{m-2}{m}} \, dV \, .
\]
Upon choosing again $ m=2q/(q-1) $, we observe that this is exactly \eqref{int-q}. Hence, interpolating as in \eqref{significativa} and taking advantage of \eqref{LC-2-bis}, we find 
\begin{equation}\label{lap-vR-5}
\begin{aligned}
\int _{\Mb^n} u^q \, \phi_R\,dV \leq & \, \frac{C}{R^2} \left[ V(B_{2R}(o)) \right]^{\frac{q-1}{q}}\left(\int_{B_{2R}(o) \setminus B_R(o)} u^q \, \phi_R \, dV\right)^{\frac{1}{q}} \\
\leq & \, \widehat{C} R^{[\alpha(n-1)+1]\frac{q-1}{q}-2} \left(\int_{B_{2R}(o) \setminus B_R(o)} u^q \, \phi_R \, dV\right)^{\frac{1}{q}} ,
\end{aligned}
\end{equation}
for another $ \widehat{C}>0 $ independent of $R \ge 2r_o$. Having established \eqref{lap-vR-5}, the thesis follows exactly as in the last part of the proof of Theorem \ref{thm: supersol}.
\end{proof}

\subsection{Existence of nonnegative nontrivial supersolutions for $q >\tilde 2_\alpha$} 

\begin{proof}[Proof of Theorem \ref{soprasol-modelli}]
As stated, our candidate supersolution is purely radial, having the form 
     \begin{equation*}
         w(r) = \frac{A}{\left(B+r^2\right)^{\frac{1}{q-1}}}
     \end{equation*}
     for suitable constants $ A,B>0 $ to be determined. Hence, asking that $ w $ is a supersolution to \eqref{LE} amounts to the validity of 
     \begin{equation}\label{rad-supsol}
     -w''(r) - (n-1) \, \frac{\psi'(r)}{\psi(r)} \, w'(r) \ge w^q(r) \qquad \forall r>0 \, .
     \end{equation}
Let us fix $ \varepsilon>0 $ so small that 
\begin{equation}\label{eps-small}
 (\alpha-\varepsilon)(n-1) + 1 -\frac{2q}{q-1} > 0 \, ,
\end{equation}
which is feasible since $ q >\tilde 2_\alpha $. At the same time, owing to \eqref{bound-lap}, we can choose $ r_\varepsilon $ large enough such that 
     \beq\label{bound-lap-eps}
\frac{\psi'(r)}{\psi(r)}  \ge \frac{\alpha - \varepsilon}{r} \qquad \forall r \ge r_\varepsilon \, .
     \eeq
We will now consider \eqref{rad-supsol} in the regions $ \left[r_\varepsilon, +\infty \right) $ and $ \left( 0 , r_\varepsilon \right) $ separately. By direct computations, we have
$$
w'(r) = -\frac{2}{q-1} \cdot \frac{A \, r}{\left(B+r^2\right)^{\frac{q}{q-1}}} \, , \qquad w''(r) = -\frac{2}{q-1} \cdot \frac{A}{\left(B+r^2\right)^{\frac{q}{q-1}}} + \frac{4q}{(q-1)^2} \cdot \frac{A \, r^2}{\left(B+r^2\right)^{\frac{q}{q-1}+1}} \, .
$$
As a result, after some algebraic simplifications, we find that $w$ is a supersolution in $ \left[r_\varepsilon, +\infty \right) $ provided
$$
  \frac{2}{q-1}\left(1 - \frac{2q}{q-1} \cdot \frac{r^2}{B+r^2} + (\alpha - \varepsilon)(n-1)\right) \ge A^{q-1} \qquad \forall r \ge r_\varepsilon \, ,
$$
where we used \eqref{bound-lap-eps} combined with the fact that $ w' \le 0 $; the previous inequality is implied by 
   \beq\label{choice-A}
\frac{2}{q-1} \left[  1 -\frac{2q}{q-1} +(\alpha-\varepsilon)(n-1) \right] \ge  A^{q-1} \, .
\eeq 
By virtue of \eqref{eps-small}, it is possible to select $A$ so small that \eqref{choice-A} holds. Next, let us deal with the inner region $ \left(0,r_\varepsilon\right) $. In order for \eqref{rad-supsol} to be satisfied there, it is enough to ask 
   \beq\label{choice-B}
  \frac{2}{q-1} - \frac{4q}{(q-1)^2} \cdot \frac{r^2}{B+r^2}  \ge A^{q-1} \qquad \forall r \in \left(0,r_\varepsilon\right) ,
\eeq
where we neglected the second term on the left-hand side since it is nonnegative, thanks to $w'\le 0$ and the assumption $ \psi' \ge 0 $. Noticing that \eqref{choice-B} holds if and only if 
   \beq\label{choice-C}
\frac{2}{q-1} - \frac{4q}{(q-1)^2} \cdot \frac{r_\varepsilon^2}{B+r_\varepsilon^2}  \ge A^{q-1} \, ,
\eeq
it is plain that one can finally choose $B$ large enough (depending only on $ q,r_\varepsilon $) and $A$ small enough (depending only on $ n,\alpha,q,\varepsilon $) in such a way that both \eqref{choice-A} and \eqref{choice-C} hold.
\end{proof}

Before proceeding with the proof of Theorem \ref{prop: str count}, we need to recall the following counterpart of Proposition \ref{LC}, dealing with curvature bounds from above (see \emph{e.g.}~\cite[Subsection 2.2]{GMV} and references therein). 
\begin{proposition}[Theorem 1.16 in \cite{MRS} and subsequent Remark]\label{LC-ch}
Let $ \Mb^n $ be a Cartan-Hadamard manifold, and let $o \in \Mb^n$. Suppose that 
\begin{equation}\label{sec-above}
       \mathrm{Sec}_{\omega}(x) \leq - \frac{\psi''(r)}{\psi(r)}\qquad \forall x \in \Mb^n \setminus \{ o \} 
    \end{equation}
   for all planes $ \omega \subset T_x(\Mb^n) $ containing the radial vector $dr$,  for some $ \psi \in C^2([0,+\infty)) $, satisfying \eqref{hp 1 psi}, such that $ \psi''(r)/\psi(r) $ can be continuously extended to $r=0$. Then 
 \begin{equation}\label{sec-above-2}
\Delta r = m(r,\theta) \geq (n-1) \, \frac{\psi'(r)}{\psi(r)} \qquad  \forall x \equiv (r,\theta) \in  (0,+\infty)\times \mathbb{S}^{n-1} \, . 
     \end{equation}
\end{proposition} 
\begin{proof}[Proof of Theorem \ref{prop: str count}]
To start with, let us prove that under the running assumptions on $ \Mb^n $ we have that for every $ \varepsilon>0 $ there exists $ r_\varepsilon>0$ such that
     \begin{equation}\label{m:maggiore}
         m(r,\theta) \geq \frac{(\alpha-\varepsilon)(n-1)}{r} \qquad   \forall (r,\theta) \in [ r_\varepsilon  ,+\infty) \times \mathbb{S}^{n-1}  \, .
     \end{equation}
To this end, we exploit Proposition \ref{LC-ch} using as a comparison function $ \psi $ the solution of the ODE problem \eqref{ODEg}, in this case with the following choice of $G$: 
$$
G(r) := 
\begin{cases}
  0  & \text{if } r \in [0,r_\delta] \, ,\\
  \frac{(\alpha-\delta)(\alpha-\delta-1)}{4r_\delta^3}(r-r_\delta) & \text{if } r \in (r_\delta,2r_\delta) \, , \\
  \frac{(\alpha-\delta)(\alpha-\delta-1)}{r^2}  & \text{if } r \ge 2r_\delta \, .
\end{cases}
$$
Because $ \mathrm{Sec}_\omega \le 0 $ everywhere ($\Mb^n$ is by assumption a Cartan-Hadamard manifold), it is readily seen that  \eqref{curvaturasezpiccola} implies that \eqref{sec-above} holds with such a $\psi$, whence \eqref{sec-above-2}. On the other hand, just like \eqref{psi-approx}, we notice that 
$$
\psi(r) = a_1 \, r^{\alpha-\delta} + a_2 \, r^{-\left( \alpha-\delta-1 \right)}  \qquad \forall r \ge 2r_\delta 
$$
for some $ a_1 > 0 $ and $ a_2 \in \R $, which in particular yields
\beq \label{psi-delta}
\lim_{r \to +\infty} \frac{\psi'(r) \, r}{\psi(r)} = \alpha-\delta \, .
\eeq 
By taking, for instance, $ \delta=\varepsilon/2 $, from \eqref{sec-above-2} and \eqref{psi-delta} we easily obtain \eqref{m:maggiore}. In light of this, now we claim that there exists a suitable choice of $A,B>0$ such that the (radial) function 
     \begin{equation*}
         w(x):=\frac{A}{\left[B+r^2(x) \right]^{\frac{1}{q-1}}}
     \end{equation*}
    is a supersolution to \eqref{LE}. Note that $ w $ is, although in a different geometric setting, the same function introduced in the proof of Theorem \ref{soprasol-modelli}. Recalling formula \eqref{L-B}, we have that $w$ is a supersolution provided
      \begin{equation}\label{supsol-m}
     -w''(r) -  m(r,\theta) \, w'(r) \ge w^q(r) \qquad \forall (r,\theta) \in (0,+\infty) \times \mathbb{S}^{n-1} \, .
    \end{equation}
Again, because $ w' \le 0 $ and the above chosen function $\psi$ is increasing, the validity of \eqref{supsol-m} is in turn implied by 
         \begin{equation*}\label{supsol-m-1}
     -w''(r)  \ge w^q(r) \qquad \forall r \in (0,r_\varepsilon) 
    \end{equation*}
    and  (consequence of \eqref{m:maggiore}) 
      \begin{equation*}\label{supsol-m-2}
       -w''(r) - \frac{(\alpha-\varepsilon)(n-1)}{r} \, w'(r) \ge w^q(r) \qquad \forall r \ge r_\varepsilon \, .
 \end{equation*}
From this point on, the proof is exactly the same as that of Theorem \ref{soprasol-modelli}.
\end{proof}


\section{The slightly subcritical regime}\label{sec: slight}

Firstly, we recall the following statement concerning Sobolev inequalities and related embeddings for one-dimensional weighted spaces. This is a crucial ingredient in the proof of Theorem \ref{da capire}. To this end, we let $ L^p\!\left([0,+\infty);\psi^{n-1}\right) $ denote the  Lebesgue space associated with the weighted measure $ \psi^{n-1} \, dr $ over $ [0,+\infty) $, and $ \left\| \cdot \right\|_{p,\psi^{n-1}} $ the corresponding norm. 
\begin{proposition}[Theorems 6.2 and 7.4 in \cite{OK}]\label{KO}
    Let $ \psi $ be a model function and let $ p \in (2,\infty) $. Then the Sobolev-type inequality
    \[
       \left( \int_0^{+\infty} \left| f \right|^p \psi^{n-1} \, dr \right)^{\frac 1 p} \le C_{n,p,\psi}  \left( \int_0^{+\infty} \left| f' \right|^2 \psi^{n-1} \, dr \right)^{\frac 1 2} \qquad \forall f \in C^1_c\!\left([0,+\infty)\right)
       \]
    holds for some positive constant $ C_{n,p,\psi} $ if and only if 
        \begin{equation}\label{sobo-type-kufner}
       \sup_{r \in (0,+\infty)}  \left( \int_0^r \psi^{n-1} \, ds \right)^{\frac 1 p} \left( \int_r^{+\infty} \frac{1}{\psi^{n-1}} \, ds \right)^{\frac 1 2} < + \infty \, .
    \end{equation}
    Moreover, the embedding of $ C^1_c([0,+\infty)) $, endowed with the norm $ \left\| f' \right\|_{2,\psi^{n-1}}  $, into $ L^p\!\left([0,+\infty);\psi^{n-1}\right) $ is compact if and only if 
            \begin{equation}\label{sobo-type-kufner-cpt}
       \lim_{r \to 0}  \left( \int_0^r \psi^{n-1} \, ds \right)^{\frac 1 p} \left( \int_r^{+\infty} \frac{1}{\psi^{n-1}} \, ds \right)^{\frac 1 2} = \lim_{r \to +\infty}  \left( \int_0^r \psi^{n-1} \, ds \right)^{\frac 1 p} \left( \int_r^{+\infty} \frac{1}{\psi^{n-1}} \, ds \right)^{\frac 1 2}=0 \, .
    \end{equation}
\end{proposition}

From the above result, one can immediately deduce the following.
\begin{corollary}\label{KO-cor}
Let $\alpha>1$. Let $ \Mb^n $ be a model manifold associated with a model function $\psi$ satisfying \eqref{hp 00 psi}. Then \eqref{sobo-type-kufner} holds for all $ p \in \left[ 2^*_\alpha , 2^* \right] $ and \eqref{sobo-type-kufner-cpt} holds for all $ p \in \left( 2^*_\alpha , 2^* \right) $. In particular, the embedding $ D^{1,2}_{\mathrm{rad}}(\Mb^n) \hookrightarrow  L^p(\Mb^n) $ is continuous for all $ p \in \left[2^*_\alpha,2^*\right] $ and compact for all $ p \in \left(2^*_\alpha,2^*\right) $. 
\end{corollary}

For $\alpha>1$, we focus at first on the range $q \in \left(2^*_\alpha-1,2^*-1\right)$.
\begin{proof}[Proof of Theorem \ref{da capire}]
Our first goal is to prove that the embedding $D^{1,2}_{\mathrm{rad}}(\Mb^n) \hookrightarrow L^{p}(\Mb^n)$ is compact for all $ p \in \left( 2^*_\alpha , 2^* \right)  $. To this end, we wish to apply Proposition \ref{KO}, and in turn this means that we need to check the validity of \eqref{sobo-type-kufner} and \eqref{sobo-type-kufner-cpt}. In fact, we go over the proof of \cite[Lemma 4.1]{MuRo}, showing that 
\begin{equation}\label{MR-rev}
 \left( \int_0^r \psi^{n-1} \, ds \right)^{\frac 1 p} \left( \int_r^{+\infty} \frac{1}{\psi^{n-1}} \, ds \right)^{\frac 1 2} \le  \left( \int_0^r \sigma^{n-1} \, ds \right)^{\frac 1 p} \left( \int_r^{+\infty} \frac{1}{\sigma^{n-1}} \, ds \right)^{\frac 1 2} \qquad \forall r>0 \, . 
\end{equation}
We provide details for the reader's convenience. First of all, by integrating \eqref{eq-MR} over $ (s,r) $ we obtain
$$
\log\!\left( \frac{\psi(r)}{\psi(s)} \right) \ge \log\!\left( \frac{\sigma(r)}{\sigma(s)} \right) \qquad \forall r>s>0 \, ,
$$
that is 
\begin{equation}\label{ord-2}
\frac{\psi(r)}{\psi(s)} \ge \frac{\sigma(r)}{\sigma(s)} \qquad \forall r>s>0 \, . 
\end{equation}
Since $ \psi(s) \sim \sigma(s) \sim s $ as $s \to 0$, we deduce that
\begin{equation}\label{ord-1}
    \psi(r) \ge \sigma (r) \qquad \forall r>0 \, .
\end{equation}
Taking advantage of \eqref{ord-1}, \eqref{ord-2}, and $ p>2 $, we then have: 
\begin{equation*}\label{ord-3}
\begin{aligned}
  \left( \int_0^r \psi^{n-1}(s) \, ds \right)^{\frac 1 p} &\left( \int_r^{+\infty} \frac{1}{\psi^{n-1}(s)} \, ds \right)^{\frac 1 2} 
 \\= & \left[\psi(r)\right]^{-\frac{(n-1)(p-2)}{2p}} \left( \int_0^r \left[\frac{\psi(s)}{\psi(r)} \right]^{n-1} ds \right)^{\frac 1 p} \left( \int_r^{+\infty} \left[\frac{\psi(r)}{\psi(s)} \right]^{n-1} ds \right)^{\frac 1 2} \\
 \le & \left[\psi(r)\right]^{-\frac{(n-1)(p-2)}{2p}} \left( \int_0^r \left[\frac{\sigma(s)}{\sigma(r)} \right]^{n-1} ds \right)^{\frac 1 p} \left( \int_r^{+\infty} \left[\frac{\sigma(r)}{\sigma(s)} \right]^{n-1} ds \right)^{\frac 1 2} \\
 = & \left[\frac{\sigma(r)}{\psi(r)}\right]^{\frac{(n-1)(p-2)}{2p}} \left( \int_0^r \sigma^{n-1}(s) \, ds \right)^{\frac 1 p} \left( \int_r^{+\infty} \frac{1}{\sigma^{n-1}(s)} \, ds \right)^{\frac 1 2} \\
 \le &  \left( \int_0^r \sigma^{n-1}(s) \, ds \right)^{\frac 1 p} \left( \int_r^{+\infty} \frac{1}{\sigma^{n-1}(s)} \, ds \right)^{\frac 1 2} ,  
\end{aligned}
\end{equation*}
that is, \eqref{MR-rev}. Given the asymptotic assumption \eqref{hp 00 psi} on $\sigma$, we are therefore in a position to appeal to Proposition \ref{KO} and Corollary \ref{KO-cor}, which in particular ensure the claimed compact embedding.
Now, by means of a classical variational argument, it is a routine argument to show that the Sobolev quotient defined in \eqref{SobolevFunct} admits a nonnegative and nontrivial minimizer for all $ q \in \left(2^*_\alpha-1,2^*-1\right)$. Any such a minimizer $u$, suitably normalized, is a nonnegative solution to the Lane-Emden equation \eqref{LE} on $\Mb^n$; moreover, by the strong maximum principle, we have that $u>0$ on $\Mb^n$, thus it is indeed a solution to \eqref{LE}. More details on this rather standard construction will be provided in the proof of Theorem \ref{non-uniq-loc} (in an analogous local setting), where it is necessary to carefully keep of all of the quantities involved.
\end{proof}

Now we focus on Theorem \ref{exi-crit}, which concerns the critical case $q=2^*_\alpha-1$. Notice that the previous argument fails, since the embedding 
$D^{1,2}_{\mathrm{rad}}(\Mb^n) \hookrightarrow L^{2^*_\alpha}(\Mb^n)$ is not compact.

 \begin{proof}[Proof of Theorem \ref{exi-crit}]
    We begin by taking a minimizing sequence $ \{ \varphi_k \} \subset C_c^1([0,+\infty)) $ of the Sobolev quotient $R_\psi= R_{\psi,2^*_\alpha-1}$, that is 
    \[
    I_k:= R_{\psi}(\varphi_k) = \frac{\|\varphi_k'\|_2}{\|\varphi_k\|_{2^*_\alpha}} \to \inf_{f \in D^{1,2}_{\mathrm{rad}}(\Mb^n) \setminus \{0\}} R_{\psi}(f) = I_\psi,
    \]
    where $I_\psi := I_{\psi,2^*_\alpha-1}$. Without loss of generality, the sequence $\varphi_k$ can be chosen to be nonnegative and satisfying
    \beq \label{norm-grad}
    \int_0^{+\infty} \left| \varphi'_k \right|^2 \psi^{n-1} \, dr = 1 \qquad \forall k \in \mathbb{N} \, ,
    \eeq
so that, in particular,
$$
\left\| \varphi_k \right\|_{L^{2^*_\alpha}(\Mb^n)} \le I_\psi^{-1} \qquad \forall k \in \mathbb{N} \, .
$$
By the compactness of local Sobolev embeddings, we can find a subsequence (that we do not relabel) and a nonnegative function $ u \in D^{1,2}_{\mathrm{rad}}(\Mb^n) $ such that 
\begin{equation}\label{loc-emb-1}
    \lim_{k \to \infty} \varphi_k = u \qquad \text{in } L^{2^*_\alpha}_{\mathrm{loc}}(\Mb^n) \, .
\end{equation}
Our goal is to prove that the above convergence is in fact \emph{global}, that is 
\begin{equation}\label{loc-emb-glob}
    \lim_{k \to \infty} \varphi_k = u \qquad \text{in } L^{2^*_\alpha}(\Mb^n) \, ;
\end{equation}
from \eqref{loc-emb-glob}, the thesis will follow. Indeed, by the definition of $ \{ \varphi_k \} $ we have
\begin{equation}\label{I-k}
1=\left\| \varphi_k' \right\|_{L^2(\Mb^n)} = I_k \left\| \varphi_k \right\|_{L^{2^*_\alpha}(\Mb^n)} \qquad \forall k \in \mathbb{N} \, ,
\end{equation}
where $ I_k $ is a positive sequence converging to the optimal Sobolev quotient $ I_\psi$. Passing to the limit as $k \to \infty$ in the above inequality, using \eqref{loc-emb-glob} and lower semicontinuity on the left-hand side, yields 
$$
\left\| u' \right\|_{L^2(\Mb^n)} \le I_\psi \left\| u \right\|_{L^{2^*_\alpha}(\Mb^n)} ,
$$
which shows that $ u $ is actually a Sobolev optimizer. Since $u \ge 0$ and $ u \not \equiv 0 $ (simple consequence of \eqref{norm-grad}), we can infer as in the last part of the proof of Theorem \ref{da capire} that it is a positive radial solution to \eqref{LE} (up to a multiplication by a positive constant).  

Hence, let us focus on proving \eqref{loc-emb-glob}. Thanks to \eqref{loc-emb-1} and the fact that $u \in L^{2^*_\alpha}(\Mb^n)$, it is not difficult to deduce that for every $ \varepsilon>0 $ there exists $ R_\varepsilon>0 $ large enough such that
\begin{equation}\label{limsup-eps}
\limsup_{k \to \infty} \left\| \varphi_k \right\|_{L^{2^*_\alpha}(B_R \setminus B_{R_\varepsilon})} \le \varepsilon \qquad \forall R \ge R_\varepsilon \, .
\end{equation}
Indeed, $ u \in L^{2^*_\alpha}(\Mb^n) $ ensures that there exists $ R_\varepsilon>0 $ such that
\beq \label{outer-mass}
\left\| u \right\|_{L^{2^*_\alpha}\left(B_{R_\varepsilon}^c\right)} < \varepsilon \, .
\eeq
Now, let $R \ge R_\eps$ be arbitrarily chosen; by combining \eqref{loc-emb-1}, \eqref{outer-mass}, and the triangle inequality, we end up with
$$
\limsup_{k \to \infty} \left\| \varphi_k \right\|_{L^{2^*_\alpha}(B_R \setminus B_{R_\varepsilon})} \le \limsup_{k \to \infty} \left\| \varphi_k - u \right\|_{L^{2^*_\alpha}(B_R \setminus B_{R_\varepsilon})} + \left\| u \right\|_{L^{2^*_\alpha}(B_R \setminus B_{R_\varepsilon})}  \le \varepsilon \, ,
$$
that is, \eqref{limsup-eps}.

Now we argue by contradiction, using a concentration-compactness argument; specifically, we will rule out vanishing and mass splitting. If \eqref{loc-emb-glob} fails, then there exists $ \lambda >0 $ such that (up to subsequences that again we do not relabel)
\begin{equation}\label{vanish}
 \liminf_{k \to \infty} \int_{R}^{+\infty} \varphi_k^{2^*_\alpha} \, \psi^{n-1} \, dr  \ge  \lambda \qquad \forall R>0 \, .
\end{equation}
Without loss of generality, we may assume that $ R_\varepsilon $ in \eqref{limsup-eps} is so large that
\begin{equation}\label{kappa-eps}
\left( \kappa - \varepsilon \right) r^\alpha \le \psi(r) \le \left( \kappa + \varepsilon \right) r^\alpha \qquad \forall r \ge R_\varepsilon \, ,
\end{equation}
and at the same time that $ k $ is sufficiently large that
\begin{equation}\label{k-large}
  \left\| \varphi_k \right\|_{L^{2^*_\alpha}(B_{2 R_\varepsilon} \setminus B_{R_\varepsilon})} \le 2\varepsilon \, .
\end{equation}
Let $ \{ \phi_R \} $ be the same family of cutoff functions as in \eqref{v-cutoff}, with $ m=1 $, and consider the following splitting of $\varphi_k$:
\beq\label{def-split}
\varphi_k = u_k + v_k := \phi_{R_\varepsilon} \, \varphi_k + (1-\phi_{R_\varepsilon}) \, \varphi_k \, . 
\eeq
We have:
\begin{equation}\label{young}
\begin{aligned}
\left(\varphi'_k\right)^2 = \left(u'_k\right)^2 + \left(v'_k\right)^2 + 2 u_k' v_k' = & \left(u'_k\right)^2 + \left(v'_k\right)^2 + 2 \left[  \phi'_{R_\varepsilon}\varphi_k + \phi_{R_\varepsilon}\varphi_k' \right] \left[  -\phi'_{R_\varepsilon}\varphi_k + (1-\phi_{R_\varepsilon})\varphi_k' \right] \\
\ge & \left(u'_k\right)^2 + \left(v'_k\right)^2 -2\left(\phi'_{R_\varepsilon}\right)^2 \varphi_k^2-  2\left|\phi'_{R_\varepsilon} \varphi_k \varphi_k'\right| \\
\ge & \left(u'_k\right)^2 + \left(v'_k\right)^2 - \left(2 + \tfrac 2 \delta \right) \left(\phi'_{R_\varepsilon}\right)^2 \varphi_k^2 - 2\delta \left(\varphi'_k\right)^2 ,
\end{aligned}
\end{equation}
where we have introduced a further free parameter $ \delta>0 $, to be chosen, via Young's inequality. Integrating \eqref{young} over $ \Mb^n $, recalling \eqref{norm-grad} and the usual decay and support properties of $ \phi_{R_\varepsilon} $, we obtain:
$$
1=\int_0^{+\infty} \left( \varphi'_k \right)^2 \psi^{n-1} \, dr \ge  \int_0^{+\infty} \left( u'_k \right)^2 \psi^{n-1} \, dr +  \int_0^{+\infty} \left( v'_k \right)^2 \psi^{n-1} \, dr - \frac{C}{R_\varepsilon^2} \cdot \frac{1+\delta}{\delta} \, \int_{R_\varepsilon}^{2R_\varepsilon} \varphi_k^2 \, \psi^{n-1} \, dr -2\delta \, ,
$$
for some $C>0$ independent of $ k,\varepsilon,\delta $ (whose specific value may change from line to line in the sequel). By H\"older's inequality, we get:
$$
\int_{R_\varepsilon}^{2R_\varepsilon} \varphi_k^2 \, \psi^{n-1} \, dr \le \left[V(B_{2R_\varepsilon})\right]^{1-\frac{2}{2^*_\alpha}} \left(  \int_{R_\varepsilon}^{2R_\varepsilon} \varphi_k^{2^*_\alpha} \, \psi^{n-1} \, dr \right)^{\frac{2}{2^*_\alpha}} ;
$$
recalling the asymptotic properties of $\psi$ and the definition of $ 2^*_\alpha $, we find that the first factor on the right-hand side of behaves exactly like $ R_\varepsilon^2 $, therefore, in view of \eqref{k-large} we may write 
$$
1 \ge  \int_0^{+\infty} \left( u'_k \right)^2 \psi^{n-1} \, dr +  \int_0^{+\infty} \left( v'_k \right)^2 \psi^{n-1} \, dr - C \, \frac{1+\delta}{\delta} \, \varepsilon^2  -2\delta \, ,
$$
Choosing $ \delta=\varepsilon $ gives
\begin{equation*}\label{key-ineq-1}
1 \ge  \int_0^{+\infty} \left( u'_k \right)^2 \psi^{n-1} \, dr +  \int_0^{+\infty} \left( v'_k \right)^2 \psi^{n-1} \, dr - C \varepsilon \, .
\end{equation*}
Using the definitions of $I_\psi$ and $I_\alpha$, also observing that $ v_k $ is supported where \eqref{kappa-eps} holds, we may deduce that 
\begin{equation}\label{key-ineq-2}
1 \ge  I_\psi^2 \left\| u_k \right\|_{L^{2^*_\alpha}(\Mb^n)}^2 +  I_\alpha^2 \left( \frac{\kappa-\varepsilon}{\kappa} \right)^{n-1} \left( \frac{\kappa}{\kappa+\varepsilon} \right)^{\frac{2(n-1)}{2^*_\alpha}} \left\| v_k \right\|_{L^{2^*_\alpha}(\Mb^n)}^2 - C \varepsilon \, .
\end{equation}
Next, by \eqref{def-split}, it follows that 
$$
 u_k^{2^*_\alpha} + v_k^{2^*_\alpha} \le \varphi_k^{2^*_\alpha} \le u_k^{2^*_\alpha} + v_k^{2^*_\alpha} + \chi_{B_{2R_\varepsilon}\setminus B_{R_\varepsilon}} \, \varphi_k^{2^*_\alpha} \, ,
$$
which, upon integration, becomes (using \eqref{I-k} and \eqref{k-large})
$$
I_k^{-2^*_\alpha} = \left\| \varphi_k \right\|_{L^{2^*_\alpha}(\Mb^n)}^{2^*_\alpha} \le \left\| u_k \right\|_{L^{2^*_\alpha}(\Mb^n)}^{2^*_\alpha} + \left\| v_k \right\|_{L^{2^*_\alpha}(\Mb^n)}^{2^*_\alpha} + (2\varepsilon)^{2^*_\alpha}  \, .
$$
Substituting such an inequality into \eqref{key-ineq-2} yields 
\begin{equation}\label{key-ineq-3}
1 \ge  I_\psi^2 \left[I_k^{-2^*_\alpha} - \left\| v_k \right\|_{L^{2^*_\alpha}(\Mb^n)}^{2^*_\alpha} - (2\varepsilon)^{2^*_\alpha} \right]^{\frac{2}{2^*_\alpha}} +  I_\alpha^2 \left( \frac{\kappa-\varepsilon}{\kappa} \right)^{n-1} \left( \frac{\kappa}{\kappa+\varepsilon} \right)^{\frac{2(n-1)}{2^*_\alpha}} \left\| v_k \right\|_{L^{2^*_\alpha}(\Mb^n)}^2 - C \varepsilon \, .
\end{equation}
Thanks to \eqref{vanish}, we infer that (up to subsequences)
$$
\lambda \le \lambda_\varepsilon := \lim_{k \to \infty} \left\| v_k \right\|_{L^{2^*_\alpha}(\Mb^n)}^{2^*_\alpha} ,
$$
so that taking $k \to \infty$ in \eqref{key-ineq-3} gives 
\begin{equation}\label{key-ineq-4}
1 \ge  I_\psi^2 \left[I_\psi^{-2^*_\alpha} - \lambda_\varepsilon - (2\varepsilon)^{2^*_\alpha} \right]^{\frac{2}{2^*_\alpha}} +  I_\alpha^2 \left( \frac{\kappa-\varepsilon}{\kappa} \right)^{n-1} \left( \frac{\kappa}{\kappa+\varepsilon} \right)^{\frac{2(n-1)}{2^*_\alpha}} \lambda_\varepsilon^{\frac{2}{2^*_\alpha}} - C \varepsilon \, .
\end{equation}
Now we may pass to the limit in \eqref{key-ineq-4} as $\varepsilon \to 0$ (along a sequence), noticing that $\{\lambda_\varepsilon\}$ converges to some positive limit, which we keep denoting by $\lambda$, such that $ \lambda \le I^{-2^*_\alpha}_\psi $; this yields, in combination with the assumed strict inequality $ I_\alpha > I_\psi $, 
\begin{equation*}\label{key-ineq-5}
1 >  I_\psi^2 \left[I_\psi^{-2^*_\alpha} - \lambda\right]^{\frac{2}{2^*_\alpha}} +  I_\psi^2 \,\lambda^{\frac{2}{2^*_\alpha}} \, .
\end{equation*}
However, the above inequality is inconsistent with the strict subadditivity of the function $ x \mapsto x^{\frac{2}{2^*_\alpha}} $. A contradiction having been reached, \eqref{loc-emb-glob} holds and therefore the proof is complete.     
\end{proof}

\subsection{Examples of model functions satisfying the assumptions of Theorems \ref{da capire} and \ref{exi-crit}}\label{sub: examples exi-crit}

\begin{example}
    Let us pick $ \psi $ as any model function satisfying
    $$
\frac{\psi'(r)}{\psi(r)} \ge \frac{\alpha}{r}  + G'(r) \qquad \text{for all $r$ large enough} \, ,
    $$
    for some smooth and globally bounded function $G(r)$. Then, it is not difficult to check that one can construct another model function $ \sigma $ such that 
    $$
\frac{\psi'(r)}{\psi(r)} \ge \frac{\sigma'(r)}{\sigma(r)} \quad \forall r>0 \qquad \text{and} \qquad \frac{\sigma'(r)}{\sigma(r)}  =\frac{\alpha}{r} + G'(r) \quad \text{for all $r$ large enough} \, ,
    $$
    so that all of the assumptions of Theorem \ref{da capire} are fulfilled. For instance, a function $\psi$ of the type
    $$
    \psi(r) = r^\alpha \, \log r \, e^{\sin r}  \qquad \text{for all $r$ large enough}
    $$
    will do, even if its asymptotics as $r \to +\infty$ is oscillatory and faster than $ r^\alpha $.
\end{example}

\begin{example}
Given $R>0$, let us consider any model function $ \psi \equiv \psi_R $ having the following properties:
$$
\psi_R(r) = r \quad \forall r \in [0,R] \, , \qquad \psi_R(r)=r^\alpha \quad \forall r > 2R \, .
$$ 
Note that, given the convexity of both $ r \mapsto r $ and $r \mapsto r^\alpha$, one can always construct $\psi_R$ in such a way that it is globally convex, hence $\mathbb{M}^n$ is Cartan-Hadamard. On one side, we have that $I_\alpha>0$, thanks to \cite[Theorem 6.2]{OK}. Furthermore, since $ 2^*_\alpha < 2^* $, it is well known that $D^{1,2}_{\mathrm{rad}}(\R^n)$ does not embed into $L^{2^*_\alpha}(\R^n)$, which is equivalent to 
\[
\lim_{R \to +\infty} \ \inf_{\varphi \in C^1_c([0,R)) \setminus \{ 0 \} } \frac{\left(\int_0^{R} \left| \varphi'(r) \right|^2 r^{n-1} \, dr \right)^{\frac 12}}{\left(\int_0^{R} |\varphi(r)|^{2^*_\alpha} \, r^{n-1} \, dr \right)^{\frac{1}{2^*_\alpha}}} = 0 \, .
\]
Thus, recalling the definitions of $I_\psi$, we can choose $R$ so large that $I_\psi < I_\alpha$, and hence all of the assumptions of Theorem \ref{exi-crit} are met.
\end{example}

\section{The intermediate regime}
\label{sec: int}
We prove Theorem \ref{thm: non-ex int} first. Afterwards, we provide examples of functions $\psi$ satisfying \eqref{hp 4 psi}. Finally, we prove Theorems \ref{thm: ex int} and \ref{non-uniq-loc}.

\subsection{Non-existence of solutions under assumption \eqref{hp 4 psi}}\label{nonex}
Throughout this subsection, we take for granted that $\Mb^n$ is a (noncompact) model manifold associated with a model function $\psi$. We aim at showing that, in such a setting, the Lane-Emden equation \eqref{LE} has no positive radial solutions for any $q \in \left(\tilde 2_\alpha, 2^*_\alpha-1\right]$ provided $\psi$ complies with \eqref{hp 4 psi}. Upon writing the equation in polar coordinates (recall \eqref{lap-rad}), this amounts to showing that for any $a>0$ the \emph{local solution} to the Cauchy problem 
\beq\label{pb Cauchy}
\begin{cases}
\displaystyle -\frac{1}{\psi^{n-1}(r)} \left(\psi^{n-1}(r) \, u'(r)\right)' = \left|u(r)\right|^{q-1}  u(r) & \text{for $r >0$} \, ,\\
 u(0) = a>0 \, , \\
 u'(0) = 0 \, ,
\end{cases}
\eeq
is either not globally defined or sign changing. Note that the existence of a unique local solution follows by standard ODE theory (recall that $\psi(r) \sim r$ as $r\to 0$), so that we can introduce the maximal existence interval $I \equiv [0,R) $ of $u$ (with possibly $R=+\infty$ when $u$ is global).
Now, let us define the \emph{energy function}
\[
F_u(r):= \frac12 \left|u'(r)\right|^2 + \frac{1}{q+1} \left|u(r)\right|^{q+1},
\]
along with the \emph{Pohozaev function}
\begin{equation}\label{poho}
  P_u(r) := \left( \int_0^r \psi^{n-1}\,ds\right) F_u(r) + \frac{\psi^{n-1}(r)}{q+1} \, u(r) \,  u'(r) \, ,
\end{equation} 
for all $ r \in I $. We recall the following basic facts from \cite{MuSo} (see Subsection 3.1 therein):
\begin{itemize}
\item[($i$)] If $u>0$ in $[0,R)$, then $u'<0$ in $(0,R]$; this plainly follows by integrating the differential equation in \eqref{pb Cauchy}.
\item[($ii$)] It holds that
\beq\label{der P u}
P_u'(r) = \psi^{n-1}(r) \left(\frac12+\frac1{q+1}-(n-1)\,\frac{\psi'(r)}{\psi^{n}(r)} \int_0^r \psi^{n-1} \,ds\right) \left|u'(r)\right|^2 .
\eeq
\end{itemize}

From now on, we argue by contradiction and suppose that a global positive solution $u$ to \eqref{pb Cauchy} exists. Note that, under assumption \eqref{hp 4 psi}, the Pohozaev function $P_u$ is nondecreasing by \eqref{der P u}. Moreover, we also claim that it cannot be constant.

\begin{lemma}\label{lem: P_u cresce}
Let $ q \in ( 1 , 2^*-1 ) $. If $u$ is a global positive solution to \eqref{pb Cauchy}, then $P_u$ is non-constant,  hence there exists $\ell \in (0,+\infty]$ such that $P_u(r) \to \ell$ as $r \to +\infty$.
\end{lemma}
\begin{proof}
Since $  |u'(r)|>0 $ for all $ r>0 $, from \eqref{der P u} it is plain than $ P_u $ is constant if and only if  
$$
\frac12+\frac1{q+1}-(n-1)\,\frac{\psi'(r)}{\psi^{n}(r)} \int_0^r \psi^{n-1} \,ds = 0 \qquad \forall r>0 \, .
$$
However, we have already observed in Remark \ref{rem: on hp 4} that this quantity is strictly positive for $r$ small.
\end{proof}

Now the idea is to obtain some crucial decay estimates for global positive solutions. The method of proof is similar to that of \cite[Theorem 1.4]{MuSo}, although no Cartan-Hadamard assumption is needed here.


\begin{lemma}\label{lemma3.2}
Let $q >\tilde 2_\alpha $ and assumption \eqref{hp 00 psi} hold. If $u$ is a global positive solution to \eqref{pb Cauchy}, then 
there exists $C>0$ such that 
\begin{equation}\label{lestime}
\begin{split}
u(r) \leq C \left(r +1 \right)^{-\frac{2}{q-1}} \quad \text{and} \quad  \left| u'(r) \right| \leq C \left(r +1\right)^{-\frac{2}{q-1}-1} \qquad \forall r \ge 0 \, .
\end{split}
\end{equation}
\end{lemma}
\begin{proof}
Integrating the equation in \eqref{pb Cauchy} over $(0,r)$, and taking into account that $u'(0)=0 = \psi(0)$, we have
\begin{equation}\label{id-1}
-u'(r) = \frac{1}{\psi^{n-1}(r)} \, \int_0^r u^q \, \psi^{n-1} \, ds \qquad \forall r>0 \, ;
\end{equation}
on the other hand, since $ u $ is decreasing, from \eqref{id-1} it follows that
\begin{equation*}\label{id-2}
-u'(r) \ge \frac{u^q(r)}{\psi^{n-1}(r)} \, \int_0^r\psi^{n-1} \, ds \qquad \forall r>0 \, ,
\end{equation*}
which can be written as
\begin{equation}\label{id-3}
\frac{1}{q-1} \left[ u^{1-q}(r) \right]' \ge \frac{1}{\psi^{n-1}(r)} \, \int_0^r\psi^{n-1} \, ds \qquad \forall r>0 \, .
\end{equation}
Integrating in turn \eqref{id-3} over $(0,r)$, and using \eqref{hp 00 psi} on the right-hand side, it is straightforward to obtain the left estimate in \eqref{lestime}. In order to prove the right estimate as well, it is convenient to go back to \eqref{id-1} and exploit the just proven bound on $u$, which yields
\begin{equation}\label{id-1-bis}
-u'(r) \le \frac{C^q}{\psi^{n-1}(r)} \, \int_0^r \left(s+1 \right)^{-\frac{2q}{q-1}} \psi^{n-1}(s) \, ds \le \frac{C}{(r+1)^{\alpha(n-1)}} \, \int_0^r \left(s+1 \right)^{-\frac{2q}{q-1}+\alpha(n-1)} ds \qquad \forall r>1 \, ,
\end{equation}
where in the last inequality we have taken advantage of \eqref{hp 00 psi}, and $C>0$ is a generic constant as in the statement. Now, we notice that
$$
-\frac{2q}{q-1}+\alpha(n-1) > -1 \qquad \iff \qquad q>\tilde 2_\alpha \, ,
$$
so that the claimed estimate for $|u'|$ follows by a direct integration of the rightmost term in \eqref{id-1-bis}.
\end{proof}

At this point, it is not difficult to complete the proof of Theorem \ref{thm: non-ex int} when $q<2^*_\alpha-1$. To this end, it is convenient to explicitly recall that, by virtue of \eqref{hp 00 psi},
\beq\label{stima int}
\frac{\kappa_1^{n-1} }{\alpha(n-1)+1} \, r^{\alpha(n-1)+1} \le \int_0^r \psi^{n-1}\, ds \le \frac{\kappa_2^{n-1} }{\alpha(n-1)+1} \, r^{\alpha(n-1)+1} \qquad \text{for all $r$ large enough} \, .
\eeq

\begin{proof}[Proof of Theorem \ref{thm: non-ex int} for $q<2^*_\alpha-1$]
Assume by contradiction that a positive radial solution $u$ to \eqref{LE} exists. Under \eqref{hp 4 psi}, we have already observed that the Pohozaev function $P_u$ defined in \eqref{poho} is nondecreasing; moreover, thanks to  Lemma \ref{lem: P_u cresce}, it is non-constant. Thus, on the one hand we have  $0=P_u(0) <P_u(+\infty) = \ell$. 
On the other hand, the decay estimates of Lemma \ref{lemma3.2} and \eqref{stima int} imply that
\[
\begin{split}
\left(\int_0^ r \psi^{n-1}\,ds\right) \left|u'(r)\right|^2 &\le C \, r^{\alpha(n-1)+1-\frac{4}{q-1}-2} \, , \\
\left(\int_0^ r \psi^{n-1}\,ds\right) u^{q+1}(r) & \le C \, r^{\alpha(n-1)+1-\frac{2(q+1)}{q-1}} \, ,\\
\psi^{n-1}(r) \, u(r) \left|u'(r)\right| & \le C \,  r^{\alpha(n-1)-\frac{4}{q-1}-1} \, ,
\end{split}
\] 
for all $r>0$ large enough. The three exponents on the right-hand sides are all equal and negative for $q<2^*_\alpha-1$. Therefore, taking limits as $r \to +\infty$ we obtain  $ \ell=0$, which is the desired contradiction.
\end{proof}

The critical case $q=2^*_\alpha-1$ is much more delicate, and we will need several preliminary lemmas to handle it. Note that the previous argument fails, since the estimates obtained in Lemma \ref{lemma3.2} only allow us to show that the Pohozaev function $P_u$ is \emph{bounded} (in fact each of the three summands forming $ P_u $ is bounded), which is \emph{a priori} compatible with the fact that $0=P_u(0) <P_u(+\infty) = \ell$. 

\begin{lemma}\label{lemma2.4garcia}
Let $q=2^*_\alpha-1$ and assumptions \eqref{hp 3 psi} and \eqref{hp 4 psi} hold. Suppose that $u$ is a global positive solution to \eqref{pb Cauchy} such that the following limit exists:
 \begin{equation}\label{L}
    L:=\lim_{r\to + \infty} r^{\frac{\alpha(n-1)-1}{2}} \, u(r) \in [0,+\infty) \, .
\end{equation}
Then 
 \begin{equation}\label{limiteu'}
      \lim_{r\to +\infty}r^{\frac{\alpha(n-1)+1}{2}} \left|u'(r) \right|= \left(\frac{\alpha(n-1)-1}{2}\right)L \, .
 \end{equation}
\end{lemma}
\begin{proof}
It suffices to prove that the limit on the left-hand side of \eqref{limiteu'} exists; afterwards, the thesis will follow by l'H\^{o}pital's rule. To this end, we observe that assumption \eqref{hp 3 psi}, the existence of both $L$ and the limit of $P_u(r)$ as $r \to +\infty$, imply in turn the existence of 
\begin{equation}\label{mu}
        \mu:= \lim_{r\to +\infty} r^{\alpha(n-1)+1} \left[ \frac{|u'(r)|^2}{2(\alpha(n-1)+1)}  +\frac{ 1}{2^*_\alpha} \cdot \frac{u(r) \, u'(r)}{r}\right].
    \end{equation}
By Lemma \ref{lemma3.2}, we have that $\mu$ is finite and  any limit point of $r \mapsto r^{\frac{\alpha(n-1)+1}{2}} \left| u'(r) \right|$ is a nonnegative real number. Let $r_k \to +\infty$ be any sequence such that
\[
\lim_{k \to \infty} r_k^{\frac{\alpha(n-1)+1}{2}} \left| u'(r_k)\right| = \lambda \ge 0 \, .
\]
By taking the limit in \eqref{mu} along the sequence $\{r_k\}$, using \eqref{L}, we can assert that the number $\lambda$ is a solution to 
\beq \label{eq-2-g}
\mu = \frac{1}{2(\alpha(n-1)+1)} \, \lambda^2 - \frac{1 }{2^*_\alpha} \, L \, \lambda \, .
\eeq
This equation has at most two distinct solutions (say $\lambda_1 < \lambda_2$); therefore, if the limit in \eqref{limiteu'} did not exist, we would have at least two sequences $s_k, t_k \to +\infty$ such that
\[
 s_k^{\frac{\alpha(n-1)+1}{2}} \left| u'(s_k) \right| \to \lambda_1 \quad \text{and} \quad t_k^{\frac{\alpha(n-1)+1}{2}} \left|u'(t_k)\right| \to \lambda_2 \qquad \text{as } k \to \infty \, .
\]
However, by the continuity of $ r \mapsto r^{\frac{\alpha(n-1)+1}{2}} \left| u'(r) \right| $, any value $ \lambda_3 \in (\lambda_1,\lambda_2) $ would also be a limit point and therefore a third solution to \eqref{eq-2-g}, a contradiction. This means that the limit in \eqref{limiteu'} must exist, which, as observed, completes the proof.
\end{proof}

\begin{lemma}\label{lem: existence L}
Let $q=2^*_\alpha-1$ and assumptions \eqref{hp 3 psi} and \eqref{hp 4 psi} hold. If $u$ is a global positive solution to \eqref{pb Cauchy}, then the limit defined in \eqref{L} does exist and $ L \in (0,+\infty) $. 
\end{lemma}
\begin{proof}
Since $P_u(r) \to \ell>0$ as $r \to +\infty$,  also recalling \eqref{stima int}, we have that there exists some $C>0$ such that
\beq\label{P dal basso}
r^{\alpha(n-1)+1} \left(\left|u'(r)\right|^2 + u^{2^*_\alpha}(r) \right) \geq C \qquad \text{for all $r$ large enough} \, .
\eeq
Next, we claim that 
\begin{equation}\label{stima dal basso}
    \liminf_{r \to +\infty} r^{\frac{\alpha(n-1)-1}{2}} \, u(r)>0 \, .
\end{equation}
In order to prove it, let us first show that 
\begin{equation}\label{limsup}  
    \limsup _{r\to +\infty} r^{\frac{\alpha(n-1)-1}{2}} \, u(r)>0 \, . 
\end{equation}
Indeed, if by contradiction \eqref{limsup} fails, we have 
\[
    \lim_{r\to +\infty} r^{\frac{\alpha(n-1)-1}{2}} \, u(r) = 0 \, .
\]
Therefore, from \eqref{P dal basso} we can infer that 
\begin{equation*}\label{lodevocontraddire}
 r^{\alpha(n-1)+1}\left|u'(r)\right|^2 \geq \frac C 2 \qquad \text{for all $r$ large enough} \, ,
\end{equation*} 
but on the other hand Lemma \ref{lemma2.4garcia} (in the special case $L=0$) ensures that 
\[
r^{\alpha(n-1)+1} \left|u'(r)\right|^2 \to 0 \qquad \text{as $r \to +\infty$} \, ,
\]
a contradiction. Now that \eqref{limsup} is proved, we concentrate on establishing \eqref{stima dal basso}. To such a purpose, let us assume by contradiction that the $\liminf$ is equal to $0$. This fact, together with \eqref{limsup}, entails that there exists a sequence $r_k\to +\infty$ of local minimum points for the function $r^{\frac{\alpha(n-1)-1}{2}} u(r)$ such that
\begin{equation}\label{ciportaaunassurdo}
    r_k^{\frac{\alpha(n-1)-1}{2}} \, u(r_k) \to 0 \qquad \text{as $k\to \infty$} \, ;
\end{equation}
in particular, the first and second order conditions for local minima ensure that, for all $ k \in \mathbb{N}$, 
\begin{equation}\label{u'u}
   0= \left. \left( r^{\frac{\alpha(n-1)-1}{2}} \, u(r)\right)'\right|_{r=r_k} = r_k^{\frac{\alpha(n-1)-1}{2}} \left[  u'(r_k) +\frac{\alpha(n-1)-1}{2} \cdot \frac{u(r_k)}{r_k} \right] 
\end{equation}
and 
\begin{equation}\label{u''u}
\begin{aligned}
   0 \le & \left. \left( r^{\frac{\alpha(n-1)-1}{2}} \, u(r)\right)''\right|_{r=r_k} \\
   = & \, r_k^{\frac{\alpha(n-1)-1}{2}} \left[  u''(r_k) +\left(\alpha(n-1)-1\right)  \frac{u'(r_k)}{r_k} + \frac{\left(\alpha(n-1)-1\right)\left(\alpha(n-1)-3\right)}{4} \cdot \frac{u(r_k)}{r_k^2}\right] . 
   \end{aligned}
\end{equation}
Hence, formulas \eqref{ciportaaunassurdo} and \eqref{u'u} yield 
\begin{equation*}
\begin{gathered}
r_k^{\alpha(n-1)+1} \left(\left|u'(r_k)\right|^2 + u^{2^*_\alpha}(r_k) \right) = \frac{\left(\alpha(n-1)-1\right)^2}{4}
 \left(r_k^{\frac{\alpha(n-1)-1}{2}} \, u(r_k) \right)^2 + \left(r_k^{\frac{\alpha(n-1)-1}{2}}\, u(r_k)\right)^{2^*_\alpha} \to 0
\end{gathered}
\end{equation*}
as $k \to \infty$, which is however inconsistent with \eqref{P dal basso}, and \eqref{stima dal basso} is thereby proved.

Estimate \eqref{stima dal basso} having been established, it remains to prove the existence of $L$ (recall from Lemma \ref{lemma3.2} that the $\limsup$ in \eqref{limsup} is always bounded). Again by contradiction, suppose that this is not the case. Upon reasoning as before, also exploiting \eqref{stima dal basso}, we can assert that there exists a sequence $r_k\to +\infty$ such that 
\begin{equation}\label{limiteA}
    r_k^{\frac{\alpha(n-1)-1}{2}} \, u(r_k) \to A \in (0,+\infty) \qquad \text{as $k\to \infty$} 
\end{equation}
and, in addition, \eqref{u'u}--\eqref{u''u} hold. Now, we reframe the differential equation in \eqref{pb Cauchy} as follows:
        \begin{equation}\label{reframed}
            -u''(r)-(n-1) \, \frac{\psi'(r)}{\psi(r)} \,  u'(r)= u^{2^*_\alpha-1}(r) \qquad \forall r>0 \,  .
        \end{equation}
By evaluating \eqref{reframed} at $r=r_k$ and using \eqref{u'u}--\eqref{u''u}, after some algebraic manipulations we deduce that 
\begin{equation*}
   \left[ r_k^{\frac{\alpha(n-1)-1}{2}} \, u(r_k) \right]^{2^*_\alpha-2} = r_k^2 \, u^{2^*_\alpha-2}(r_k) \leq \frac{\alpha(n-1)-1}{2}\left[(n-1) \, \frac{r_k \, \psi'(r_k)}{\psi(r_k)}- \frac{\alpha(n-1)+1}{2}\right]. 
\end{equation*}
At this point, we can take the limit as $k \to \infty$, exploiting \eqref{psi'psi} (consequence of \eqref{hp 3 psi}) and \eqref{limiteA}, to conclude that
\begin{equation}\label{Aq-1}
    A^{2^*_\alpha-2}\leq \left(\frac{\alpha(n-1)-1}{2}\right)^2. 
\end{equation}
Finally, let us show that \eqref{Aq-1} entails $P_u(+\infty)<0$, which is the desired contradiction. Indeed, at first we point out that by \eqref{u'u} and \eqref{limiteA} we have  
\begin{equation}\label{directly}
\begin{aligned}
   \lim_{k \to \infty} r_k^{\alpha(n-1)+1} \, u^{2^*_\alpha}(r_k) & = A^{2^*_\alpha} \, , \\
  \lim_{k \to \infty}  r_k^{\alpha(n-1)+1} \left|u'(r_k)\right|^{2} & = \left(\frac{\alpha(n-1)-1}{2}\right)^2 A^2 \, , \\ 
     \lim_{k \to \infty}  r_k^{\alpha(n-1)} \, u(r_k) \, u'(r_k) & = - \frac{\alpha(n-1)-1}{2} \, A^2 \, .
    \end{aligned}
\end{equation}
As a result, computing $P_u$ along $r_k$, taking the limit as $k\to \infty$ (note that \eqref{stima int} holds as an asymptotic identity with $ \kappa_1=\kappa_2=\kappa $ owing to \eqref{hp 3 psi}), and exploiting \eqref{Aq-1}--\eqref{directly}, we arrive at
\begin{equation}\label{P u inf =}
\begin{aligned}
  \frac{\alpha^{n-1}}{\kappa^{n-1}} \, P_u(+\infty) & = \frac{1}{\alpha(n-1)+1} \left[ \frac12 \left(\frac{\alpha(n-1)-1}{2}\right)^2A^2 + \frac{A^{2^*_\alpha}}{2^*_\alpha} \right] -\frac{\alpha(n-1)-1}{2} \cdot \frac{A^2}{2^*_\alpha} \\
   & \le \frac{1}{\alpha(n-1)+1} \left(\frac{\alpha(n-1)-1}{2}\right)^2A^2 \left[\frac12+\frac1{2^*_\alpha}-1\right] < 0 \, ,
   \end{aligned}
\end{equation}
which is clearly incompatible with $ P_u(+\infty) = \ell>0$.
\end{proof}

\begin{lemma}\label{lemma2.5garcia}
Let $q=2^*_\alpha-1$, assumptions \eqref{hp 3 psi} and \eqref{hp 4 psi} hold, and $u$ be a global positive solution to \eqref{pb Cauchy}. Then 
    \begin{equation*}\label{L^q-1}
        L^{2^*_\alpha-2} = \left(\frac{\alpha(n-1)-1}{2}\right)^2 ,
    \end{equation*}
    where $L$ has been defined in \eqref{L}.
    \end{lemma}
    \begin{proof}
    Since $u$ solves \eqref{reframed}, multiplying by $u'$ and integrating over $(r,+\infty)$ we deduce that
        \begin{equation}\label{E1}
            -\frac{\left[u'(r)\right]^2}{2}+(n-1) \int_r^{+\infty} \left|u'\right|^2 \frac{\psi'}{\psi}\,ds = \frac{u^{2^*_\alpha}(r)}{2^*_{\alpha}} \qquad  \forall r>0 \, .
        \end{equation}
Note that, in view of the estimates of Lemma \ref{lemma3.2} and \eqref{psi'psi}, we have $[u']^2 \, \psi'/\psi \in L^1([1,+\infty))$; in particular, by l'H\^{o}pital's rule, we may infer that
   \begin{equation}\label{E2}
\lim_{r \to +\infty}  r^{\alpha(n-1)+1}\int_r^{+\infty} \left|u'\right|^2 \frac{\psi'}{\psi}\,ds = \frac{\alpha}{\alpha(n-1)+1} \left(\frac{\alpha(n-1)-1}{2}\right)^2 L^2 \, ,
         \end{equation}
where we have taken advantage of Lemmas \ref{lemma2.4garcia}--\ref{lem: existence L} as well. Now, the idea is to multiply both sides of \eqref{E1} by $r^{\alpha(n-1)+1}$ and take the limit as $r \to +\infty$; again using Lemmas \ref{lemma2.4garcia}--\ref{lem: existence L}, along with \eqref{E2}, we end up with  
       \[
       - \frac12\left(\frac{\alpha(n-1)-1}{2}\right)^2 L^2 +\frac{\alpha(n-1)}{\alpha(n-1)+1} \left(\frac{\alpha(n-1)-1}{2}\right)^2 L^2 = \frac{L^{2^*_\alpha}}{2^*_\alpha} \, ,
       \]
      which is readily seen to be equivalent to the thesis.
    \end{proof}
    
    Now we are in the position to prove Theorem \ref{thm: non-ex int} also in the critical case $q=2^*_\alpha-1$. 
\begin{proof}[Proof of Theorem \ref{thm: non-ex int} for $q=2^*_\alpha-1$]
%
%
Assume by contradiction that a positive radial solution $u$ to \eqref{LE} exists. Recall that $P_u(+\infty)=\ell>0$, thanks to Lemma \ref{lem: P_u cresce}. Taking advantage of Lemmas \ref{lemma2.4garcia}--\ref{lem: existence L}, and proceeding exactly as in the proof of the first equality in \eqref{P u inf =}, we obtain: 
\[
\begin{aligned}
0<\frac{\alpha^{n-1}}{\kappa^{n-1}} \, \ell & = \frac{1}{\alpha(n-1)+1} \left[ \frac12 \left(\frac{\alpha(n-1)-1}{2}\right)^2 L^2 + \frac{L^{2^*_\alpha}}{2^*_\alpha} \right] -\frac{\alpha(n-1)-1}{2} \cdot \frac{L^2}{2^*_\alpha}  \\
& = \frac{L^2}{2(\alpha(n-1)+1)} \left[ \left(\frac{\alpha(n-1)-1}{2}\right)^2+ \frac{\alpha(n-1)-1}{\alpha(n-1)+1} \, L^{2^*_\alpha-2} - \frac{(\alpha(n-1)-1)^2}{2}\right].
\end{aligned}
\]
However, owing to Lemma \ref{lemma2.5garcia}, this would give
\[
0<\frac{\alpha^{n-1}}{\kappa^{n-1}} \, \ell  = \frac{(\alpha(n-1) -1)^2 \, L^2}{8(\alpha(n-1)+1)} \left[1+ \frac{\alpha(n-1)-1}{\alpha(n-1)+1} -2\right] = \frac{(\alpha(n-1) -1)^2 \, L^2}{8(\alpha(n-1)+1)} \left[ \frac{2}{2^*_\alpha} - 1  \right] < 0 \, ,
\]
a contradiction.
\end{proof}

\subsection{Examples of model functions satisfying the assumptions of Theorem \ref{thm: non-ex int}}\label{sub: examples non-ex}

At first, we provide a sufficient condition for the validity of \eqref{hp 4 psi}, which turns out to be very useful for the examples we will construct.

\begin{lemma}\label{lem: suff cond non-ex}
Let $\alpha>1$ and $\psi$ be a model function satisfying
\beq\label{hp 5 psi}
\psi(r) \, \psi''(r) \le \frac{\alpha-1}{\alpha} \left[ \psi'(r)\right]^2 \qquad \forall r>0 \, .
\eeq
Then \eqref{hp 4 psi} holds.
\end{lemma}
\begin{proof}
Without loss of generality, we may assume that $\psi'>0$ everywhere (we will come back to this at the end of the proof). Therefore, we can rewrite \eqref{hp 5 psi} as
\beq\label{hp 5 psi - equiv}
\left(\frac{\psi(r)}{\psi'(r)}\right)' \ge \frac1{\alpha} \qquad \forall r >0 \, .
\eeq 
At this point, integrating by parts, we obtain
\[
\begin{split}
\int_0^r \psi^{n-1} \,ds &= \int_0^r \psi^{n-1} \, \frac{\psi'}{\psi} \, \frac{\psi}{\psi'} \, ds = \frac{\psi^{n-1}(r)}{n-1} \cdot \frac{\psi(r)}{\psi'(r)} - \int_0^r \frac{\psi^{n-1}}{n-1} \left(\frac{\psi}{\psi'}\right)' ds \\
& \le \frac{\psi^n(r)}{(n-1) \, \psi'(r)} - \frac1{\alpha(n-1)} \int_0^r \psi^{n-1}\,ds \, .
\end{split}
\]
As a result,
\[
\frac{\alpha(n-1)+1}{\alpha(n-1)}\int_0^r \psi^{n-1}\,ds \le \frac{\psi^n(r)}{(n-1) \,\psi'(r)} \, ,
\]
whence it follows that
\beq\label{poho-ex}
\frac{(n-1) \, \psi'(r)}{\psi^n(r)} \int_0^r \psi^{n-1}\,ds \le \frac{\alpha(n-1)}{\alpha(n-1)+1}= \frac12 + \frac{1}{2^*_\alpha} \, .
\eeq
Thus, condition \eqref{hp 4 psi} is satisfied for any $q\le 2^*_\alpha-1$.

To complete the proof, suppose that $\psi'$ is not everywhere positive. Then, since $\psi'(0)=1$, there exists some interval $[0,r_0]$ such that $ \psi'>0 $ in $ [0,r_0) $ and $ \psi'(r_0)=0 $. Upon arguing exactly as in the first part, we can assert that \eqref{poho-ex} holds at least in $[0,r_0]$. Now there are two possibilities: either $ \psi' \le 0 $ in $ (r_0,+\infty) $ or there exists some $ r_2>r_0 $ such that $ \psi'(r_2)>0 $. In the first case, it is apparent that \eqref{poho-ex} also holds in $ (r_0,+\infty) $. Hence, to complete the proof, let us rule out the second case. Still from the continuity of $ \psi' $, if that occurs, there necessarily exists some  $ r_1 \in [r_0,r_2) $ such that $\psi'(r_1)=0$ and $ \psi'>0 $ in $ (r_1,r_2] $. On such an interval, we can then rewrite \eqref{hp 5 psi} as
$$
\left[\log\!\left(\psi'(r)\right) \right]' \le \frac{\alpha-1}{\alpha} \left[\log\!\left(\psi(r)\right) \right]'  \qquad \forall r \in (r_1,r_2) \, ;
$$
however, a simple integration of the above inequality leads to the contradiction 
\[
+\infty = \log\!\left(\psi'(r_2)\right) - \lim_{r \downarrow r_1}\log\!\left(\psi'(r)\right) \le \frac{\alpha-1}{\alpha} \left[\log\!\left(\psi(r_2)\right) - \log\!\left(\psi(r_1)\right) \right]  . \qedhere
\]
\end{proof}

The first example we propose is a very simple explicit model.

\begin{example}
Let $\alpha>1$, and let 
\[
\psi(r) = \frac{(1+r)^\alpha}{\alpha}-\frac{1}{\alpha} \, .
\]
Then $\psi$ satisfies \eqref{hp 1 psi}, \eqref{hp 3 psi}, and \eqref{hp 4 psi} for every $q \in \left(1,2^*_\alpha-1\right]$. Indeed, with regards to the latter property, it is immediate to check that \eqref{hp 5 psi} holds, so that Lemma \ref{lem: suff cond non-ex} applies.
\end{example}

The manifold $\Mb^n$ associated with the above choice of $\psi$ is $C^1$, however, since $\psi''(0) \neq 0$, it is not $C^2$ (of course, smoothness is lost at the pole $o$ only). In what follows we provide a class of examples of $C^\infty$ Cartan-Hadamard model manifolds satisfying \eqref{hp 5 psi} in which, according to Theorem \ref{thm: non-ex int} and Lemma \ref{lem: suff cond non-ex}, we have a non-existence result for radial solutions of \eqref{LE}.

\begin{lemma}\label{c-inf}
Let $F \in C^\infty([0,+\infty))$ with $ F(0)=F'(0)=0 $. Then $ G(r):=F(r)/r^2 \in C^\infty([0,+\infty))$. 
\end{lemma}
\begin{proof}
It is plain that $ G \in C^\infty((0,+\infty)) $, so we only need to prove that $G$ and all of its derivatives are continuous down to $r=0$. By using iteratively the product rule for the derivative of a function, one can easily verify that $ G^{(k)} $, for any $ k \in \mathbb{N} \setminus\{ 0 \} $, has the following form:
\beq\label{exp-1}
G^{(k)}(r) = \sum_{j=0}^k a_j \, \frac{F^{(k-j)}(r)}{r^{j+2}} \qquad \forall r>0 \, ,
\eeq
for suitable real coefficients $ \{ a_j \}_{j=0\ldots k} $ that we do not need to compute explicitly. Because $F \in  C^\infty([0,+\infty)) $, in view of the Taylor expansions of $F$ and its derivatives centered at $r=0$, from \eqref{exp-1} we deduce that either
\beq\label{c1}
G^{(k)}(r) = o(r^m) \qquad \text{as } r \to 0^+ \, , \ \forall m \in \mathbb{N} \, ,
\eeq
or there exist a real number $ b \neq 0 $ and an integer $ J \le k+2 $ such that 
\beq\label{c2}
\lim_{r \to 0^+} r^{J} \, G^{(k)}(r)  = b \, .
\eeq
Next, let us complete the proof by induction, \emph{i.e.}, assuming that $ G^{(k-1)} $ is continuous at $r=0$, we aim at showing that $ G^{(k)} $ is also continuous at $r=0$. Note that the base case is true since, in view of the assumptions, $ G $ is continuous at $r=0$, with $ G(0) = F''(0)/2 $. As concerns the induction step, there are two possibilities. If \eqref{c1} holds, then it is apparent that $ G^{(k)} $ is continuous at $r=0$. If, on the contrary, \eqref{c2} holds, we need to make sure that $ J $ cannot be positive; on the other hand, if $ J $ were a positive integer, then 
$$
\left[G^{(k-1)}\right]'(r) = G^{(k)}(r) \sim \frac{b}{r^J} \qquad \text{as } r \to 0^+ \, ,
$$
which is clearly incompatible with the assumed continuity of $ G^{(k-1)} $ at $r=0$. Hence, $J$ is necessarily nonpositive, ensuring that $ G^{(k)} $ is continuous at $r=0$ as desired.
\end{proof}

\begin{example}\label{eee}
Integrating twice the differential equation
\begin{equation}\label{condizione}
    \left(\frac{\psi(r)}{\psi'(r)}\right)' =\frac{1}{\alpha} + \frac{\alpha -1}{\alpha}\cdot\frac{1}{1+r^2}
\end{equation}
we get 
\begin{equation}\label{psi}
    \psi (r)= r\cdot\exp\left[(\alpha-1)\int_0^r\frac{s-\arctan s}{s[s+(\alpha-1)\arctan s]} \, ds\right],
\end{equation}
and it is easy to check that \eqref{condizione} implies \eqref{hp 5 psi}. Also, a model manifold associated with such a function $\psi$ turns out to be Cartan-Hadamard: indeed, using again \eqref{condizione}, it holds 
\begin{equation*}
     \mathrm{Sect}_{\omega}(r)= -\frac{\psi''(r)}{\psi(r)} = -\left(\frac{\psi'(r)}{\psi(r)}\right)^2\cdot\frac{\alpha-1}{\alpha}\cdot \frac{r^2}{1+r^2}\leq 0 \, .
\end{equation*}
Furthermore, $\psi$ is $C^\infty([0,+\infty))$ and it complies with both \eqref{hp 1 psi'} for all $k \in \mathbb{N}$ and \eqref{hp 3 psi} (we will justify these properties more in general in the next example). Hence, the model manifold $ \mathbb{M}^n $ associated with such a $\psi$ is of class $C^\infty$ and satisfies all of the hypotheses of Theorem \ref{thm: non-ex int}.
\end{example}

\begin{example}
Now we aim at generalizing Example \ref{eee}. To this end, let us consider, in the place of \eqref{psi}, the following type of model function: 
\begin{equation*}\label{psigenerale}
     \psi (r)= r\cdot\exp\left[(\alpha-1)\int_0^r\frac{s-f(s)}{s\left[s+(\alpha-1) f(s)\right]} \, ds\right],
\end{equation*}
where $ f \in C^\infty([0,+\infty)) $ is an arbitrary  satisfying the following assumptions:
\beq\label{cond-f-1}
f^{(2k)}(0)=0 \quad \forall k \in \mathbb{N} \, , \qquad f'(0)=1 \, , \qquad 0 \le f'(r) \le 1 \quad \forall r \ge 0 \, , 
\eeq
and 
\beq\label{cond-f-2}
\qquad f(r) = o(r) \quad \text{as } r \to +\infty  \, , \qquad \int_1^{+\infty} \frac{f(s)}{s^2} \, ds < +\infty  \, .
\eeq
First of all, let us verify that $ \psi \in C^\infty([0,+\infty)) $. To this end, it is readily seen that both the functions
$$
r \mapsto r-f(r) \qquad \text{and} \qquad  r \mapsto r\left[r+(\alpha-1) f(r)\right]
$$
satisfy the hypotheses of Lemma \ref{c-inf}, whence
$$
F(r) := \frac{r-f(r)}{r^2} \in C^\infty([0,+\infty)) \qquad \text{and} \qquad G(r):= \frac{r \left[r+(\alpha-1) f(r)\right]}{r^2} \in C^\infty([0,+\infty)) \, .
$$
Furthermore, since $G(r) > 0$ for all $r \ge 0$ (simple consequence of $ f'(0) = 1 $ and the monotonicity of $f$), we can deduce that
$$
H(r) := \frac{r-f(r)}{r\left[r+(\alpha-1) f(r)\right]} = \frac{F(r)}{G(r)} \in  C^\infty([0,+\infty)) \, ,
$$
which clearly yields the claim. Also, from the leftmost assumption in \eqref{cond-f-1} it is not difficult to check that \eqref{hp 1 psi'} holds for all $m \in \mathbb{N}$, so that the model manifold $\mathbb{M}^n$ associated with $\psi$ is indeed smooth. Next, let us confirm that the hypotheses of Theorem \ref{thm: non-ex int} are met. An explicit computation gives
\begin{equation*}
    \left(\frac{\psi(r)}{\psi'(r)}\right)' =\frac{1}{\alpha} + \frac{\alpha -1}{\alpha} \, f'(r) \, ,
\end{equation*}
which, owing to $f' \le 0$, ensures the validity of \eqref{hp 5 psi - equiv} (and therefore of \eqref{hp 5 psi}). In addition, the condition $ f' \le 1 $ yields
\begin{equation*}
    \mathrm{Sect}_{\omega}(r)= -\frac{\psi''(r)}{\psi(r)} = \left(\frac{\psi'(r)}{\psi(r)}\right)^2\cdot\frac{\alpha-1}{\alpha}\left[-1+f'(r) \right] \leq 0 \, ,
\end{equation*}
so that $ \mathbb{M}^n $ is actually a Cartan-Hadamard model. Finally, as a consequence of \eqref{cond-f-2}, one can easily check that \eqref{hp 3 psi} is also met.
\end{example}
 
\subsection{Existence of solutions on certain Cartan-Hadamard manifolds}
In this subsection, we fix once for all $\alpha>1$ and $q \in \left(\tilde 2_\alpha, 2^*_\alpha-1 \right)$. Our strategy of proof of Theorem \ref{thm: ex int} consists in explicitly constructing a model function $\psi$, that is, a positive function $ \psi \in C^\infty([0,+\infty)) $ satisfying \eqref{hp 1 psi} and \eqref{hp 1 psi'}, such that \eqref{hp 3 psi} holds and \eqref{LE} admits at least a radial solution on the associated manifold $\Mb^n$. Moreover, we are able to make sure that $ \psi $ is convex, \emph{i.e.}~$ \Mb^n $ is a \emph{Cartan-Hadamard} manifold. The proof is divided into three main steps. The first one, which is the content of the next lemma, provides the basic idea of the construction, but provides only a \emph{Lipschitz} continuous model function, which needs to be carefully smoothened in the subsequent steps. 

\begin{lemma}\label{lemmaLip}
There exist a locally Lipschitz continuous and convex function $\bar \psi:[0,+\infty) \to [0,+\infty)$, which is positive in $ (0,+\infty) $,  $C^\infty$ in $\left[0,\bar r_0\right) \cup \left(\bar r_0, +\infty\right)$ for some $\bar r_0>0$, satisfies \eqref{hp 1 psi}, \eqref{hp 1 psi'}, and \eqref{hp 3 psi}, and a corresponding $C^1$ function $\bar u:[0,+\infty) \to (0,+\infty)$ such that 
\beq\label{eq psi lip}
-\bar u'' - (n-1) \, \frac{\bar \psi'}{\bar \psi} \, \bar u' = \bar u^q \qquad \text{in $(0,+\infty) \setminus \{\bar r_0\}$} \, ,
\eeq
with $\bar u'(0)=0$ and $\bar u(0)>0$.
\end{lemma}

\begin{proof}
Let $w_1$ be the unique positive solution to 
\beq \label{eq:w1}
-w_1'' - (n-1) \, \coth(r) \, w_1' = w_1^q \qquad \text{in $(0,1)$} \, ,
\eeq
with $w_1'(0)=0$ and $w_1(1)=0$, namely $w_1$ is the positive radial solution of the homogeneous Dirichlet problem \eqref{loc-dir} in the geodesic unit ball of the hyperbolic space (for the existence and uniqueness of such a solution we refer to \cite{MaSa} or \cite{BeFeGr}). Also, let 
\begin{equation}\label{w_2}
 w_2(r):= \frac{c}{(r-r_0)^{\frac{2}{q-1}}} \qquad \text{for $r>r_0$} \, ,   
\end{equation}
where $r_0 \in \R $ will be determined in the sequel and
\[
c \equiv c\!\left(\alpha, n, q\right) := \left[\frac{2}{q-1} \left( \alpha(n-1)-\frac{q+1}{q-1}\right) \right]^\frac{1}{q-1} > 0 \, .
\]
We point out that $r_0$, hence also $r$, may be negative in \eqref{w_2}. By a direct calculation, one finds that $w_2$ satisfies
\beq\label{eq:w2}
-w_2'' - (n-1) \, \frac{\alpha}{(r-r_0)} \, w_2' = w_2^q \qquad \text{in $\left(r_0,+\infty\right)$} \, .
\eeq
The idea is to glue $w_1$ and $w_2$ in a convenient way. To this end, we observe that by definition the graphs of $w_1$ and $w_2$ do not intersect for $r_0 \ge 1 $; on the other hand, since $ w_2 $ vanishes uniformly in $ (0,+\infty) $ as $ r_0 \to -\infty $, they do intersect if $r_0 \ll -1$. Thus, we have that 
\[
\tilde r_0:= \inf \left\{r_0 \in \R: \ \text{the graphs of $w_1$ and $w_2$ do not intersect}\right\}
\]
is a real number. Upon choosing $r_0=\tilde r_0$ in \eqref{w_2}, from the definition of infimum we infer that the graphs of $w_1$ and $w_2$ necessarily intersect \emph{tangentially} at some point $\bar r_0>\max\{\tilde r_0,0\}$ (see Figure \ref{w1-w2}). 
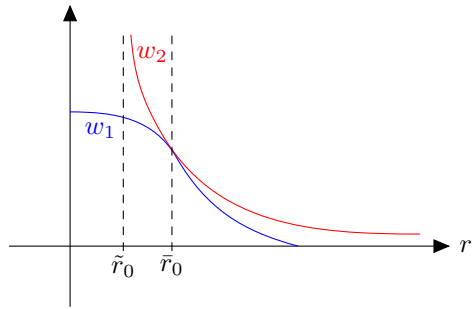
\begin{figure}[ht]
\begin{tikzpicture}[line cap=round,line join=round,>=triangle 45, scale=2]
\draw [->] (-0.4,0)-- (2.5,0) node [right] {$r$};
\draw [->] (0,-0.4)-- (0,1.6) node [right] {};
\draw[color=black] (0.2,0.77) node {{\color{blue}$w_1$}};
\draw [blue]  (0,0.89) to [out=0, in=120] (0.7,0.588) to [out=300, in=165] (1.5,0.0);
\draw[color=red] (0.54,1.27) node {{\color{red}$w_2$}};
\draw [red]  (0.4,1.4) to [out=275, in=120] (0.6,0.75) 
to [out=300, in=180] (2.3, 0.08);
		\node [below] at (0.67,0) {$\bar{r}_0$};
        \draw [dashed] (0.67,-0.03)-- (0.67,1.4) node [right] {};
        \node [below] at (0.35,0) {$\tilde r_0$};
         \draw [dashed] (0.35,-0.03)-- (0.35,1.4) node [right] {};
\end{tikzpicture}
   \caption{\small A representation of the graphs of $w_1$ (blue) and $w_2$ (red).}\label{w1-w2}
\end{figure}
This implies that the function
\begin{equation*}
    \bar u(r):= \begin{cases} w_1(r) & \text{if $r \in [0,\bar r_0]$} \, , \\
w_2(r) & \text{if $r \in \left(\bar r_0,+\infty\right)$} \, ,
\end{cases}
\end{equation*}
is at least of class $C^1([0,+\infty))$, and moreover
\begin{equation}\label{disderivate}
 \bar u''\!\left(\bar r_0^-\right) := w_1''(\bar r_0) \le w_2''(\bar r_0) = \bar u''\!\left(\bar r_0^+\right)  ,
\end{equation}
with a possible \emph{positive} jump for the second derivative. Note that the jump cannot be negative, otherwise the graph of $w_2$ would lie strictly below that of $w_1$ in a right neighborhood of $ \bar r_0 $, allowing for further intersections when $ r_0 $ is slightly larger than $ \tilde{r}_0 $, which is again a contradiction to the definition of $ \tilde{r}_0 $. Now, due to \eqref{eq:w1} and \eqref{eq:w2}, we observe that $\bar u$ satisfies \eqref{eq psi lip} with
\beq \label{psi-k}
\bar \psi(r):= 
\begin{cases}
\sinh(r) & \text{if } r \in [0,\bar r_0] \, , \\
\kappa (r-\tilde r_0)^\alpha & \text{if } r \in \left( \bar r_0 , +\infty \right)  ,
\end{cases}
\eeq 
for any $\kappa>0$. Also, $\bar u'(0) = w_1(0)= 0 $ and $\bar u(0)=w_1(0)>0$. Clearly, we can choose $\kappa$ in such a way that $\bar \psi$ is continuous in $[0,+\infty)$. Moreover, we note that \eqref{eq psi lip} can be written as follows: 
\begin{equation*}\label{barpsi'barpsi}
   \frac{\bar \psi'}{\bar \psi} = -\frac{\bar u^q+\bar u''}{(n-1) \, \bar u'} \qquad \text{in $(0,+\infty) \setminus \{\bar r_0\}$} \, ;
\end{equation*}
in particular, by virtue of \eqref{disderivate} and the fact that $ \bar u $ is $C^1$ with $\bar u'<0$, it follows that
\[
\bar \psi'\!\left(\bar r_0^-\right) \le \bar \psi'\!\left(\bar r_0^+\right) ,
\]
\emph{i.e.}\ also $ \bar \psi' $ has (at most) a positive jump at $\bar r_0$. To sum up, $ \bar \psi $ is continuous in $ [0,+\infty) $, positive in $ \left( 0 ,+\infty \right) $, $C^\infty $ in $ \left[ 0 , \bar r_0 \right) \cup \left( \bar r_0 , +\infty \right)  $, satisfies \eqref{hp 1 psi}, \eqref{hp 1 psi'}, and \eqref{hp 3 psi} (due to \eqref{psi-k}), and $\bar \psi'$ is continuous and increasing in $ \left[0,\bar r_0 \right) \cup \left(\bar r_0,+\infty\right)$ with at most a positive jump at $\bar r_0$; as a result, $\bar \psi$ is (locally) Lipschitz continuous and convex in the whole of $[0,+\infty)$, as desired.
\end{proof}

In the next lemma, which is by far the most technical part of the construction, we improve the regularity of the above model function $\bar \psi$ from Lipschitz to $C^1$, by carefully removing the singularity at $ \bar r_0 $ via a suitable modification of the solution $\bar u$. It is plain to see that if in \eqref{disderivate} the identity holds then one can directly move on to the proof of Theorem \ref{thm: ex int}, as the forthcoming Lemma  \ref{lemmaintermedio} becomes redundant. For this reason, we will implicitly assume, without loss of generality, that in \eqref{disderivate} the strict inequality holds. 

\begin{lemma} \label{lemmaintermedio}
Let $\bar r_0>0$ be as in Lemma \ref{lemmaLip}. For any $\varepsilon>0$ sufficiently small there exist a $C^1$ and convex function $\bar \psi_\varepsilon:[0,+\infty)\to [0,+\infty)$, which is positive in $ (0,+\infty) $, $C^\infty$ in $\left[0,\bar r_0\right) \cup \left(\bar r_0,\bar r_0+\varepsilon \right) \cup \left(\bar r_0+\varepsilon, +\infty \right)$, satisfies \eqref{hp 1 psi}, \eqref{hp 1 psi'} and \eqref{hp 3 psi}, and a corresponding $C^2$ function $\bar u_\varepsilon:[0,+\infty)\to (0,+\infty)$ such that
    \begin{equation}\label{equeps}
        -\bar u_\varepsilon''-(n-1) \, \frac{\bar \psi_\varepsilon'}{\bar \psi_\varepsilon} \, \bar u_\varepsilon'=\bar u_\varepsilon^q \qquad \text{in $(0,+\infty)$} \, .
    \end{equation}
with $ \bar u_\varepsilon''' $ exhibiting two jumps at $ \bar r_0 $ and $ \bar r_0 + \varepsilon $. Moreover, there exists a constant $ K >0 $ such that 
\beq\label{const-conv}
\frac{\bar \psi''_\varepsilon(r)}{\bar \psi_\varepsilon(r)} \ge \frac{K}{r^2+1} \qquad \forall r \in [0,+\infty) \setminus \left\{ \bar r_0  , \bar r_0 + \varepsilon \right\} .
\eeq
\end{lemma}
\begin{proof}
Let $w_1$, $w_2$, and $\bar u$ be defined as in the proof of Lemma \ref{lemmaLip}. Before continuing, it is convenient to introduce the following change of variables:
 \[
     w_1 =: F_1^{-\frac{1}{q-1}} \qquad \text{and} \qquad w_2 =: F_2^{-\frac{1}{q-1}} \, .
 \]
Note that, given the explicit expression \eqref{w_2} of $w_2$, the function $ F_2 $ is purely quadratic: 
 \begin{equation}\label{F2}
    F_2(r):= \left[\frac{2}{q-1}\left(\alpha(n-1)-\frac{q+1}{q-1}\right)\right]^{-1} \left(r-\tilde r_0 \right)^2
    \qquad \forall r>\tilde r_0 \, ,
 \end{equation}
a property that we will exploit shortly. Thus, by defining 
 \[
     F(r):= \begin{cases} F_1(r) & \text{if $r \in [0,\bar r_0]$} \, , \\
F_2(r) & \text{if $r \in (\bar r_0,+\infty)$} \, ,
\end{cases}
 \]
it holds that
\begin{equation*}\label{ubarF}
    \bar u = F^{-\frac{1}{q-1}} \, .
\end{equation*}
In particular, due to \eqref{F2} and the fact that $\bar r_0 > \tilde r_0$, we may rewrite $F_2$ as 
\begin{equation}\label{w2alternativa}
    F_2(r)=\frac{A}{2}\,(r-\bar r_0)^2 + B\,(r-\bar r_0) +C  \qquad \forall r>\tilde r_0 \, ,
\end{equation}
for suitable positive constants $A \equiv A\!\left(\alpha,n,q\right)$, $B \equiv B \!\left(\alpha,n,q,\tilde r_0, \bar r_0 \right)$, and $ C \equiv C\!\left(\alpha,n,q,\tilde r_0, \bar r_0 \right)$. Using \eqref{w2alternativa} and the $C^1$ regularity of $\bar u$ (which is therefore inherited by $F$), we necessarily have that 
\begin{equation}\label{id-F12}
    C= F_2(\bar r_0)= F_1(\bar r_0) \qquad \text{and} \qquad B= F'_2(\bar r_0)= F'_1(\bar r_0) \, . 
\end{equation}
Moreover, we observe that \eqref{disderivate} (assumed to hold with strict inequality) and \eqref{id-F12} entail
\begin{equation}\label{disderivateequiv}
     A= F_2''(\bar r_0) < F_1''(\bar r_0)=: \tilde A \, ,
\end{equation}
hence $F'''$ exhibits a negative Dirac delta at $\bar r_0$. 

Next, we proceed by modifying $F_2$ in $ (\bar r_0 , +\infty ) $ in a proper way, without touching $ F_1 $ in $ [0,\bar r_0] $. Namely, for any $\varepsilon>0$, let us introduce the function $F_{2,\varepsilon}$ defined as 
\begin{equation*}\label{F2eps}
 F_{2,\varepsilon}(r):= \frac{\tilde A}{2}\,(r-\bar r_0)^2 +B\,(r-\bar r_0) +C -\frac{\tilde A-A}{\varepsilon} \int_{\bar r_0}^r\int_{\bar r_0}^t\int_{\bar r_0}^s \chi_{[\bar r_0,\bar r_0+\varepsilon]} (y) \,dy \,ds \,dt \qquad  \forall r> \bar r_0 \, .
\end{equation*}
From the above definition, we notice that 
\begin{equation}\label{derivataterza}
    F_{2,\varepsilon}'''(r)= -\frac{\tilde A -A}{\varepsilon}\,\chi_{[\bar r_0,\bar r_0+\varepsilon]}(r) \qquad  \forall r> \bar r_0 \, ;
\end{equation}
also, recalling \eqref{id-F12} and \eqref{disderivateequiv}, we have that 
\begin{equation}\label{siraccordaC2conF1}
F_{2,\varepsilon}(\bar r_0)= F_1(\bar r_0) \, , \qquad  F_{2,\varepsilon}'(\bar r_0)= F_1'(\bar r_0) \, , \qquad F_{2,\varepsilon}''(\bar r_0)= F_1''(\bar r_0) \, ,
\end{equation}
so that $ F_{2,\varepsilon} $ is joined to $ F_1 $ with a $C^2$ connection at $\bar r_0$. Even more so, by direct computations we can obtain the explicit expression of $F_{2,\varepsilon}$: 
\begin{equation}\label{F2esplit}
     F_{2,\varepsilon}(r)= \begin{cases} \frac{\tilde A}{2}\,(r-\bar r_0)^2 +B\,(r-\bar r_0) +C -\frac{\tilde A-A}{\varepsilon} \cdot \frac{(r-\bar r_0)^3}{6} & \text{if $r \in (\bar r_0, \bar r_0+\varepsilon]$} \, , \\
\frac{A}{2}\,(r-\bar r_0)^2 + \left(B+\frac{\tilde A -A}{2} \, \varepsilon\right) (r-\bar r_0) +C-\frac{\tilde A-A}{6} \, \varepsilon^2 & \text{if $r \in (\bar r_0+\varepsilon,+\infty)$} \, .
\end{cases}
\end{equation}
As a result, it is evident that $F_{2,\varepsilon}$ is of class $C^2$ in the whole of $(\bar r_0,+\infty)$ and thus, also taking into account \eqref{derivataterza} and \eqref{siraccordaC2conF1}, it turns out that the function 
\begin{equation}\label{Fe}
F_\varepsilon(r):=\begin{cases} F_1(r) & \text{if $r \in [0,\bar r_0]$} \, , \\
F_{2,\varepsilon}(r) & \text{if $r \in (\bar r_0,+\infty)$} \, ,
\end{cases}
\end{equation}
is of class $C^2([0,+\infty))$, with its third derivative having two jumps at $\bar r_0$ and $\bar r_0+\varepsilon$. Furthermore, owing to \eqref{F2esplit}, it is not difficult to check that $ F_{2,\varepsilon} $ is always everywhere positive.

At this point, we can introduce 
\beq \label{uFeps}
\bar u_\varepsilon 
:= 
F_\varepsilon^{-\frac{1}{q-1}} 
\eeq
to obtain a positive $C^2$ function, complying with $ \bar u_\varepsilon(0)>0 $ and $ \bar u_\varepsilon'(0)=0 $, that solves \eqref{equeps} for a model function $\bar\psi_\varepsilon$ implicitly defined by 
\begin{equation}\label{psieps}
    \frac{\bar \psi_\varepsilon'}{\bar \psi_\varepsilon}= -\frac{\bar u_\varepsilon^q+ \bar u_\varepsilon''}{(n-1) \, \bar u_\varepsilon'}= \frac{1}{n-1}\left(\frac{q-1}{F_\varepsilon'}+ \frac{q}{q-1}\cdot\frac{F_\varepsilon'}{F_\varepsilon}-\frac{F_\varepsilon''}{F'_\varepsilon}\right),
\end{equation}
and therefore of class $C^1$ (note that the denominators on the rightmost side never vanish thanks to \eqref{F2esplit} and the fact that $ F_\varepsilon = F_1  $ in $ [0,\bar r_0] $). Since $ \bar u_\varepsilon = w_1 $ in $ [0,\bar r_0] $, integrating from $ r=0 $ to an arbitrary $ r>0 $, it is plain that there exists a unique model function satisfying \eqref{psieps}, which is actually $ C^\infty $ in $ [0, \bar r_0) \cup (\bar r_0 , \bar r_0 + \varepsilon) \cup ( \bar r_0 + \varepsilon , +\infty ) $ and fulfills \eqref{hp 1 psi} and \eqref{hp 1 psi'}. In order to complete the proof, it remains to show that $\bar \psi_\varepsilon$ is convex in the whole $(0,+\infty)$ and complies with the desired growth estimate \eqref{hp 3 psi}. To this aim, we will need to pick $\varepsilon>0$ small enough.

To begin with, let us focus our attention on the interval $(\bar r_0,+\infty)$.
According to \eqref{F2esplit} and \eqref{Fe}, we can assert that there exists $\bar \varepsilon>0 $ such that 
\begin{equation}\label{beh}
  F_\varepsilon(r) = \frac{A\,r^2}{2} + O(r) \, , \quad 
F'_{\varepsilon}(r) = A\,r+O(1) \, , \quad  F''_{\varepsilon}(r) = A \qquad \text{as $r\to +\infty$, uniformly in $\eps \in (0,\bar \eps]$} \, . 
\end{equation} 
We point out that, when writing $g_\eps(r)=O(r^m)$ as $r \to +\infty$ (for some $ m \in \mathbb Z $), uniformly in $\eps \in (0,\bar \eps]$, we mean that there exist $ \mathsf{C}, R>0$ (independent of $\varepsilon$) such that
\[
\sup_{r \ge R} \, \sup_{\eps \in (0,\bar \eps]} \left|\frac{g_\eps(r)}{r^m}\right|  \le  \mathsf{C} \, .
\]
Therefore, \eqref{psieps} yields 
\begin{equation*}
\frac{\bar\psi_\varepsilon'(r)}{\bar\psi_\varepsilon(r)}= \frac{1}{n-1}\left(\frac{q-1}{A}+\frac{q+1}{q-1}\right)\frac{1}{r} +O\!\left(\frac{1}{r^2}\right) \qquad \text{as $r\to +\infty$, uniformly in $\varepsilon \in (0, \bar \varepsilon]$} \, .
\end{equation*}
Remarkably, this is equivalent to  
\begin{equation}\label{psi' su psi 8 ott}
\frac{\bar\psi_\varepsilon'(r)}{\bar\psi_\varepsilon(r)} = \frac{\alpha}{r} +O\!\left(\frac{1}{r^2}\right) \qquad \text{as $r\to +\infty$, uniformly in $\varepsilon \in (0,\bar \varepsilon ]$} \, , 
\end{equation}
since, thanks to \eqref{F2} and \eqref{disderivateequiv}, we have that
\begin{equation}\label{A}
    A= F''_2(\bar r_0) =\frac{\left(q-1\right)^2}{\alpha(n-1)\left(q-1\right)-\left(q+1\right)} \, .
\end{equation}
Upon integrating \eqref{psi' su psi 8 ott} from $ R \gg \bar r_0 $ to an arbitrary $ r>R $, we readily deduce that \eqref{hp 3 psi} holds (for some $\kappa>0$ possibly depending on $\varepsilon$). 

With a view to establishing convexity, we observe that by combining \eqref{psieps} and the identity
\begin{equation*}
    \frac{\bar \psi_\varepsilon''}{\bar \psi_\varepsilon} = \left(\frac{\bar\psi_\varepsilon'}{\bar\psi_\varepsilon}\right)' + \left(\frac{\bar\psi_\varepsilon'}{\bar\psi_\varepsilon}\right)^2 
\end{equation*}
we arrive at 
\begin{equation}\label{interminidiF}
\begin{aligned}
  (n-1)^2  \, \frac{\bar \psi_\varepsilon''}{\bar \psi_\varepsilon} = & \, \frac{\left(q-1\right)^2+n \left(F_\varepsilon''\right)^2-(n+1)\left(q-1\right)F_\varepsilon''-(n-1) \, F_\varepsilon' \, F_\varepsilon'''}{\left(F_\varepsilon'\right)^2} \\
  & + \frac{2q \, F_\varepsilon+\frac{q(n-3)}{q-1} \, F_\varepsilon \, F_\varepsilon''+\frac{q}{q-1}\left(\frac{q}{q-1}-(n-1)\right)\left(F_\varepsilon'\right)^2}{F_\varepsilon^2} \, ,
  \end{aligned}
\end{equation}
for every $r\in (0,\bar r_0)\cup (\bar r_0,\bar r_0+\varepsilon)\cup (\bar r_0+\varepsilon,+\infty)$. Next, we exploit the uniform asymptotic expansions  \eqref{beh} in \eqref{interminidiF}, along with the fact that $ F'''_\varepsilon = 0 $ in $ (\bar r_0 + \varepsilon, +\infty) $, to deduce that 
\begin{multline*}
    (n-1)^2  \, \frac{\bar \psi_\varepsilon''(r)}{\bar \psi_\varepsilon(r)} = \left[ \frac{\left(q-1\right)^2 + n A^2 - (n+1)\left(q-1\right)A}{A^2} \right. \\ \left.+ \frac{q A+ \frac{q(n-3)}{2\left(q-1\right)} \, A^2+ \frac{q}{q-1}\left(\frac{q}{q-1}-(n-1)\right)A^2}{\frac{A^2}{4}}\right] \frac1{r^2}+O\!\left(\frac1{r^3}\right)
\end{multline*}
as $r \to +\infty$, uniformly in $\eps \in (0,\bar \eps]$. The coefficient in front of $1/r^2$ on the right-hand side can be rewritten by using \eqref{A} as
\[
\begin{aligned}
&\frac{\left[\alpha (n-1)\left(q-1\right)-\left(q+1\right)\right]^2}{\left(q-1\right)^2}  + n -\frac{n+1}{q-1} \left[\alpha (n-1)\left(q-1\right)-\left(q+1\right) \right] \\ 
& \quad \qquad \qquad + \frac{4q}{\left(q-1\right)^2} \left[\alpha(n-1)\left(q-1\right)-\left(q+1\right)\right]+\frac{2q}{q-1}\,(n-3)+\frac{4q}{q-1}\left(\frac{q}{q-1}-(n-1)\right) \\
& \qquad \qquad =\frac{1}{(q-1)^2} \Big[\alpha^2 (n-1)^2 \left(q-1\right)^2 +\alpha(n-1)\left(q-1\right) \left( 4q-(n+1)\left(q-1\right)-2\left(q+1\right)\right)\\
& \qquad \qquad  \hphantom{ =\frac{1}{\left(q-1\right)^2} \Big[} \ + \left(q+1\right)^2 -\left(q+1\right) \left(4q-(n+1)\left(q-1\right)\right) + 4q^2 -2q(n+1)\left(q-1\right)+n\left(q-1\right)^2\Big] \\
& \qquad \qquad = (n-1)^2 \, \alpha(\alpha-1) \, , 
\end{aligned}
\]
whence
\[
\frac{\bar \psi_\varepsilon''(r)}{\bar \psi_\varepsilon(r)} = \frac{\alpha(\alpha-1)}{r^2} + O\!\left(\frac{1}{r^3}\right) \qquad \text{as $r\to +\infty$, uniformly in $\varepsilon \in (0,\bar \varepsilon ]$} \, .
\]
In particular, this ensures the existence of some $ R \gg \bar r_0$, independent of $ \eps \in (0,\bar \eps] $, such that
\beq\label{conv inf}
\frac{\bar \psi_\varepsilon''(r)}{\bar \psi_\varepsilon(r)} \ge \frac{\alpha(\alpha-1)}{2r^2} \qquad \forall r >R \, , \ \forall \eps  \in (0,\bar \eps] \, ,
\eeq
which yields the convexity of $\bar \psi_\eps$ in $(R,+\infty)$. 

Regarding the interval $(\bar r_0+\eps,R]$, here we can exploit the estimates
\[
\sup_{r \in (\bar r_0+\eps,R]} \left|F_{\eps}(r)-F_2(r) \right| \le  \mathsf{C} \eps \, , \qquad \sup_{r \in (\bar r_0+\eps,R]} \left|F_{\eps}'(r)-F_2'(r)\right| \le  \mathsf{C} \eps \, ,
\]
for some $ \mathsf{C}>0$ independent of $\eps \in (0,\bar \eps]$, which can be easily deduced from \eqref{F2esplit}. Moreover, we have that $F_{\eps}''=F_2''$ and $F_{\eps}'''=0$ therein. As a consequence, by \eqref{interminidiF} there exists another constant $ \mathsf{C}>0 $ (that we do not relabel), independent of $\eps \in (0,\bar \eps] $,  such that
\[
\sup_{r \in (\bar r_0+\eps,R]}\left| \frac{\bar \psi_\eps''(r)}{\bar \psi_\eps(r)} - \frac{\bar \psi''(r)}{\bar \psi(r)} \right| \le \mathsf{C} \eps \, ,
\]
which implies that
\beq\label{conv inf 2}
\frac{\bar\psi_\eps''(r)}{\bar\psi_\eps(r)} \ge \frac{\bar\psi''(r)}{\bar\psi(r)}-\mathsf{C}\eps =
\frac{\alpha(\alpha-1)}{\left( r-\tilde r_0 \right)^2} -\mathsf{C}\eps \ge \frac{\alpha(\alpha-1)}{R^2} -\mathsf{C}\eps >0 \qquad \forall r \in  (\bar r_0+\eps,R] \, ,
\eeq 
provided $ \eps>0 $ is small enough depending on $R$ as well (which is consistent since $ R $ only depended on the previous upper bound $ \bar \eps $). 

Hence, so far we have shown that for $\eps>0$ sufficiently small, the function $\bar \psi_\eps$ is convex in $(\bar r_0+\eps,+\infty)$; more precisely, it satisfies \eqref{conv inf} and \eqref{conv inf 2}. Furthermore, it is clearly convex in $(0,\bar r_0)$ as well, since it coincides with the original function $\bar \psi$ there, which is precisely the hyperbolic sine (recall \eqref{psi-k}), so that 
\beq\label{conv inf 2 bis}
\frac{\bar\psi_\eps''(r)}{\bar\psi_\eps(r)}  = 1 \qquad \forall r \in (0,\bar r_0) \, .
\eeq

It remains to determine what happens in the intermediate interval $(\bar r_0, \bar r_0+\eps)$. To this end, we preliminarily notice that 
\[
\sup_{r \in (\bar r_0,\bar r_0+\eps)} \left|F_{\eps}(r)-F_2(r)\right| \le \bar{\mathsf{C}} \eps^2 \, , \qquad \sup_{r \in (\bar r_0, \bar r_0+\eps)} \left|F_{\eps}'(r)-F_2'(r)\right| \le \bar{\mathsf{C}} \eps \, ,
\]
for some $ \bar{\mathsf{C}}>0$ independent of $\eps \in (0,\bar \eps]$, which can still be deduced from \eqref{F2esplit}. Moreover, by construction we have that
$$
F''_{\eps}(r) \in \big(A, \tilde A\big) \quad \text{and} \quad F_{\eps}'''(r) = - \frac{\tilde A-A}{\eps}  \qquad \forall r \in (\bar r_0 , \bar r_0 + \eps) \, .
$$
In particular, in such an interval $ -F_\varepsilon''' $ diverges uniformly to $ +\infty $ as $ \eps \to 0^+ $, $F'_\eps $ stays positively bounded away from zero, and all of the other quantities stay bounded. Therefore, owing to formula \eqref{interminidiF}, we see that for any $M>0$ there exists a small enough $\eps'>0$ such that 
\beq\label{conv inf 3}
\inf_{r \in (\bar r_0,\bar r_0+\eps)} \frac{\bar\psi_\eps''(r)}{\bar\psi_\eps(r)} \ge M \qquad \forall \eps \in \left(0,\eps'\right)  .
\eeq
To conclude, thanks to \eqref{conv inf}, \eqref{conv inf 2}, \eqref{conv inf 2 bis}, and \eqref{conv inf 3}, we observe that for every small enough $\eps>0$ the function $\bar \psi_\eps$, which is $C^1$ and piecewise $ C^\infty $, complies with estimate \eqref{const-conv} for a suitable constant $ K>0 $.
In particular, it is globally convex, which completes the proof.
\end{proof}

We are finally in the position to prove the main result of the present subsection, namely Theorem \ref{thm: ex int}. In order to do that, we take advantage of the existence-type result in Lemma  \ref{lemmaintermedio}, along with a further technical correction aimed at removing the discontinuity of the second derivative of $ \bar \psi_\varepsilon $. 

\begin{proof}[Proof of Theorem \ref{thm: ex int}] 
For the sake of notational simplicity, let us set $ \bar \psi \equiv \bar \psi_\eps $ and $ F \equiv F_\varepsilon $, where $ \bar \psi_\eps $ is the $ C^1 $ and piecewise $ C^\infty $ convex model function provided by Lemma \ref{lemmaintermedio}, and $ F_\varepsilon $ is related to the corresponding solution $ \bar u_\eps $ of equation \eqref{equeps} by formula \eqref{uFeps}. Note that $ F''' $ exhibits two jumps: our goal is to remove both of them by smoothening $F$ via an approximation that preserves the convexity of $\bar \psi$. However, it is  not restrictive to assume that $F'''$ has one jump only (which we will do), say at some point $r_1>0$, as the proposed approximation procedure can be carried out in two subsequent steps. 


To begin with, let us introduce the notation 
    \begin{equation*}
    \begin{aligned}
        G\!\left(F,F',F'',F'''\right):=\, &  \frac{\left(q-1\right)^2+n \left(F''\right)^2-(n+1)\left(q-1\right)F''-(n-1) \, F' \, F'''}{\left(F'\right)^2} \\
  & + \frac{2q \, F+\frac{q(n-3)}{q-1} \, F \, F''+\frac{q}{q-1}\left(\frac{q}{q-1}-(n-1)\right)\left(F'\right)^2}{F^2} \, .
  \end{aligned}
    \end{equation*}
    We can immediately observe that $G$ continuously depends on all of its arguments (where $ F,F'>0 $); in particular, it is \emph{linear} with respect to the last one. Owing to the fact that $ \psi''(r)/\psi(r) > K/\left( r^2+1 \right) $ in $ (0,r_1) \cup (r_1,+\infty) $, recalling \eqref{interminidiF} we deduce that  
    \[
     G\!\left(F(r_1),\,F'(r_1),\,F''(r_1),\,F'''\!\left( r_1^{\pm} \right)\right) >0 \, ,
    \]
    since $ F $ is $C^2\!\left( \left[ 0 , +\infty \right) \right) \cap C^\infty\!\left( \left[ 0 , r_1 \right) \cup \left( r_1 , +\infty \right)  \right)  $ and $ F''' $ has a jump at $ r_1 $. Thanks to the just observed continuity of $G$ and $ F,F'>0 $, we can assert that there exist $\varepsilon , \delta >0 $ such that 
       \[
     G\!\left(F(r_1)+t_0,\,F'(r_1)+t_1,\,F''(r_1)+t_2,\,F'''\!\left( r_1^{\pm} \right) + t_3 \right) > \varepsilon \qquad \forall t_0,t_1,t_2,t_3 \in (-\delta,\delta) \, .
    \]
Now, because the coefficient in front of $F'''$, in the expression of $G$, is given by $-(n-1)/F'(r_1)<0$, and we know that $G$ is linear in $ F''' $, it actually follows that
    \begin{equation}\label{G}
    \begin{gathered}
       G\!\left(F(r_1)+t_1, \, F'(r_1)+t_2, \,F''(r_1)+t_3, \,S\right)>\varepsilon \qquad \forall t_0,t_1,t_2 \in (-\delta,\delta) \, , \ \forall S < \mathsf{M}+\delta \, , \\ 
       \text{where } \mathsf{M} := \max\!\left\{ F'''\!\left( r_1^{+} \right) , \, F'''\!\left( r_1^{-} \right)  \right\} . 
       \end{gathered}
    \end{equation} 
    Given the properties of $F$, we are in a position to apply a suitable regularization $ \left\{ F_k \right\} $ of smooth and positive functions such that 
    \begin{equation}\label{G1}
        F_k \to F \, , \quad F_k' \to F'\, , \quad F''_k \to F'' \qquad\text{as $k\to \infty$} \, , \ \text{uniformly in } [0,+\infty) \, ,
    \end{equation}
and 
    \begin{equation}\label{G2}
        F_k''' \to F''' \qquad \text{as $k \to \infty$} \, , \ \text{uniformly in $[0,r_1-\eta] \cup [r_1+\eta,+\infty)$ for every $\eta \in (0,r_1)$} \, ;
    \end{equation}
    moreover, we can make sure that $ \eta $ is so small that 
    \beq \label{G3}
    F_k'''(r) < \mathsf{M} + \delta \qquad \forall r \in (r_1-\eta , r_1 + \eta)
    \eeq
    for all $ k \in \mathbb{N} $ large enough, since the jump discontinuity of $ F''' $ at $r_1$ guarantees that $ F''' $ strictly less than $ \mathsf{M}+\delta $ in a suitable neighborhood of $r_1$. Note that such a construction can be achieved, for instance, by combining a local cutoff argument (near $r_1$) and a mollification; in particular, it is possible to leave the original function $F$ unmodified away from $ r_1 $.   
    
    Putting together \eqref{G1} and \eqref{G2} with the fact that $G\left(F, F', F'', F'''\right)>0$ in $ (0,r_1)\cup (r_1,+\infty)$, we thus have that  
    \begin{equation}\label{G-conv-1}
        G\!\left(F_k(r), \, F_k'(r), \, F_k''(r), \,F_k'''(r)\right)>0 \qquad \forall r\in(0,r_1-\eta] \cup [r_1+\eta,+\infty) \, ,
    \end{equation}
    for every $\eta \in (0 , r_1) $ and for all $k$ sufficiently large (depending on $\eta$). In order to conclude the proof, we need to also cover the intermediate region $(r_1-\eta, r_1+\eta)$. To this end, we claim that 
    \begin{equation}\label{G-conv-2}
    G\!\left(F_k(r), \, F_k'(r), \, F_k''(r), \,F_k'''(r)\right)>0 \qquad  \forall r \in (r_1-\eta, r_1+\eta) \, ,
    \end{equation}
for any small enough $\eta>0$ and for all $k$ sufficiently large (depending on $\eta$). To show the claim, we first observe that by virtue of \eqref{G1} and $ F \in C^2 $ we can always pick $ \eta $ so small that
    \begin{equation*}
        F_k(r)-F(r_1) \, , \  F'_k(r)-F'(r_1) \, , \  F''_k(r)-F''(r_1)\in (-\delta,\delta) \qquad \forall r \in (r_1-\eta, r_1+\eta) \, .
    \end{equation*}
On the other hand, for a possibly smaller $\eta$ and a larger $k$, inequality \eqref{G3} holds. At this point, the claim readily follows from \eqref{G}. 

Now, as in \eqref{uFeps}, we set $ u := F_k^{-1/(q-1)} $ and reconstruct the model function $\psi$ upon integrating the analogue of \eqref{psieps} with $ \bar u_\varepsilon $ replaced by $u$ and $ F_\eps $ by $F_k$. In this way, we obtain a pair $ (u,\psi) $ satisfying 
$$
- u'' - (n-1) \, \frac{\psi'}{\psi} \, u' = u^q \qquad \text{in $(0,+\infty)$} \, ,
$$
where $ \psi \in C^\infty([0,+\infty)) $ is convex thanks to \eqref{G-conv-1} and \eqref{G-conv-2}; moreover, it still complies with \eqref{hp 1 psi}, \eqref{hp 1 psi'}, and \eqref{hp 3 psi} (recall that $ F_\varepsilon $ has not been modified outside of a small neighborhood of $r_1$). Finally, since $ q<2^*_\alpha-1 $ and we know that (up to constants)
$$
\psi (r) \sim r^{\alpha} \, , \quad u(r) \sim r^{-{2}/{(q-1)}} \, , \quad  u'(r) \sim r^{-{(q+1)}/{(q-1)}}  \qquad \text{as }  r \to +\infty \, ,
$$
by means of a standard cutoff argument it is not difficult to check that $ u $ also belongs to the energy space $ D^{1,2}_{\mathrm{rad}}\!\left( \Mb^n \right) $.
\end{proof}

\subsection{Non-uniqueness for the Dirichlet problem in geodesics balls}\label{sub:nonuniq}

\begin{proof}[Proof of Theorem \ref{non-uniq-loc}]
    We argue by contradiction, assuming that \eqref{loc-dir} admits a unique radial solution $u_R$ for all $R>0$. Taking it for granted from now on, we will proceed by stating and proving a number of claims.  

\medskip
\noindent \textbf{Claim 1:} \it The embedding of $ D^{1,2}_{\mathrm{rad}}\!\left( \Mb^n \right) $ into $ L^p\!\left( \Mb^n \right) $ fails for all $ p < 2^*_\alpha  $. \rm 

\smallskip

\noindent Thanks to assumption \eqref{hp 00 psi}, this is a direct consequence of Proposition \ref{KO}.


\medskip
\noindent \textbf{Claim 2:} \it For every $R>0$, the radial solution $u_R$ to \eqref{loc-dir} is the unique positive Sobolev minimizer of the quotient 
\beq \label{def-IR}
I_R:=\inf_{f \in H^1_{0,\mathrm{rad}}(B_R) \setminus \{ 0 \} } \frac{\left(\int_0^{R} \left| f' \right|^2 \psi^{n-1} \, dr \right)^{\frac 12}}{\left(\int_0^{R} |f|^{q+1} \, \psi^{n-1} \, dr \right)^{\frac{1}{q+1}}}
\eeq
such that
\beq\label{MR to IR}
\int_0^R u_R^{q+1} \, \psi^{n-1} \, dr  =: M_R = I_R^{2\frac{q+1}{q-1}} \, .
\eeq
\rm 
Since the embedding of $ H^1_0\!\left( B_R \right) $ into $ L^p\!\left(B_R\right) $ is compact for all $ p<2^* $, the existence of a positive minimizer for $I_R$ follows via the direct method of the calculus of variations, similarly to the proof of Theorem \ref{da capire}. Indeed, to this end, one needs to fix the ``mass'' 
\beq \label{sobo-mass}
\int_0^{R} |f|^{q+1} \, \psi^{n-1} \, dr = M>0
\eeq
and minimize the squared gradient norm in \eqref{def-IR} subject to \eqref{sobo-mass}, obtaining, by compactness, a certain function $ v \in H^{1}_{0,\mathrm{rad}}\!\left( B_R \right)  $ such that
\beq \label{sobo-mass-2}
\int_0^{R} |v|^{q+1} \, \psi^{n-1} \, dr = M \qquad \text{and} \qquad \int_0^{R} \left( v' \right)^2 \psi^{n-1} \, dr = I_R^2 \, M^{\frac{2}{q+1}} \, .
\eeq
Due to the structure of the functional in \eqref{def-IR}, it is plain that $ v $ can be taken nonnegative without loss of generality; moreover, the first-order optimality conditions entail that $v$ is a radial solution of  
\begin{equation}\label{loc-dir-lambda}
\begin{cases}
-\Delta v = \lambda \, v^q \, , \quad v>0  & \text{in } B_R \, , \\ 
v = 0  & \text{on } \partial B_R \, ,
\end{cases}
\end{equation}
for some $ \lambda>0 $, so that, in particular, it is $C^1 $ in $ \overline{B}_R $ and positive in $ B_R $. In fact, multiplying the differential equation in \eqref{loc-dir-lambda} by $v$ and integrating by parts, recalling \eqref{sobo-mass-2}, we find the relation 
$$
 I_R^2 \, M^{\frac{2}{q+1}} = \lambda \, M  \qquad \iff \qquad M = \lambda^{-\frac{q+1}{q-1}} \, I_R^{2 \frac{q+1}{q-1}} \, .
$$
Because $M>0$ is arbitrarily fixed, choosing it precisely as 
$$
M=M_R:=I_R^{2 \frac{q+1}{q-1}} 
$$
ensures that $ \lambda=1 $, that is, the Sobolev minimizer $v$ with mass $ M_R $ is a radial solution to \eqref{loc-dir}. On the other hand, we are assuming that \eqref{loc-dir} admits a unique radial solution, whence $ v=u_R $. Moreover, since \emph{any} positive Sobolev minimizer with mass $M_R$ satisfies  \eqref{loc-dir}, the thesis follows.

\medskip
\noindent \textbf{Claim 3:} \it The function $R \mapsto I_R$ is nonincreasing, with
\beq\label{lim-IR}
\lim_{R \to 0} I_R = +\infty \qquad \text{and} \qquad \lim_{R \to +\infty} I_R = 0 \, . 
\eeq
\rm 
The monotonicity follows directly from the definition. From Claim 1 we know that
$$
I_\infty:=\inf_{ f \in D^{1,2}_{\mathrm{rad}}(\Mb^n) \setminus \{ 0 \} } \frac{\left(\int_0^{+\infty} \left| f' \right|^2 \psi^{n-1} \, dr \right)^{\frac 12}}{\left(\int_0^{+\infty} |f|^{q+1} \, \psi^{n-1} \, dr \right)^{\frac{1}{q+1}}} = 0 \, .
$$
Thus, by the definitions of $I_R$ and $I_\infty$, it is easy to deduce that 
$$
\lim_{R \to +\infty} I_R = I_\infty \, ,
$$
whence the right identity in \eqref{lim-IR}. Next, assume by contradiction that $I_R$ has a finite limit $ \mathsf{I}>0 $ as $ R \to 0 $. Then, recalling \eqref{MR to IR}, the sequence $ \left\{ u_{1/k} \right\} $ (extended to zero outside $B_{1/k}$) would be bounded in $ H^1_0(B_1) $ with an $ L^{q+1}(B_1) $ norm convergent to some positive constant; however, this is incompatible with the compactness of the embedding $ H^1_0(B_1) \hookrightarrow L^{q+1}(B_1) $, as the a.e.~pointwise limit is $0$. Therefore, the left identity in \eqref{lim-IR} holds too.

\medskip
\noindent \textbf{Claim 4:} \it The function
\beq \label{def-AR}
A(R) := u_R(0) \qquad \forall R>0
\eeq 
is continuous. \rm 

\smallskip
\noindent First of all, we observe that the function $ R \mapsto I_R $ is continuous. Indeed, left continuity is a simple consequence of the fact that $I_R$ can equivalently be written as 
$$
I_R = \inf_{S \in (0,R)} \, \inf_{ f \in H^{1}_{0,\mathrm{rad}}(B_S) \setminus \{0\} } \frac{\left(\int_0^{S} \left| f' \right|^2 \psi^{n-1} \, dr \right)^{\frac 12}}{\left(\int_0^{S} |f|^{q+1} \, \psi^{n-1} \, dr \right)^{\frac{1}{q+1}}} \, .
$$
With regards to right continuity, take any decreasing sequence $ R_k \searrow R $, and denote by $ \{ v_k \} \subset H^1_{0,\mathrm{rad}}\!\left(B_{R_k}\right) $ the corresponding sequence of (positive) Sobolev minimizers of fixed mass $M$ (recall \eqref{sobo-mass}). Then
\beq \label{Ik-def}
I_{R_k} = \frac{\left(\int_0^{R_k} \left| v_k' \right|^2 \psi^{n-1} \, dr \right)^{\frac 12}}{\left(\int_0^{R_k} v_k^{q+1} \, \psi^{n-1} \, dr \right)^{\frac{1}{q+1}}} = \frac{\left(\int_0^{R_k} \left| v_k' \right|^2 \psi^{n-1} \, dr \right)^{\frac 12}}{M^{\frac{1}{q+1}}} \qquad \forall k \in \mathbb{N} \, ,
\eeq 
so that $ \{ v_k \} $ is bounded, say, in $ H^1_0(B_{R+1}) $ since $ I_{R_k} \le I_R $. By compactness, $ \{ v_k \} $ converges (up to subsequences) weakly in $ H^1_0(B_{R+1}) $ and thus strongly in $ L^{q+1}(B_{R+1}) $ to some $ \overline{v} $ which, by construction, actually belongs to $ H^1_{0,\mathrm{rad}}\!\left( B_R \right) $. Hence, taking the limit of \eqref{Ik-def} as $k \to \infty$, we end up with 
$$
I_R \ge  \limsup_{k \to \infty} I_{R_k} \ge  \frac{\left(\int_0^{R} \left| \overline{v}' \right|^2 \psi^{n-1} \, dr \right)^{\frac 12}}{\left(\int_0^{R} \overline{v}^{q+1} \, \psi^{n-1} \, dr \right)^{\frac{1}{q+1}}} \ge I_R \, , 
$$
which yields the thesis. Note that, \emph{a posteriori}, convergence is also strong in $ H^1_0(B_{R+1}) $. Now, let $ R \in (0,+\infty) $, take any sequence $ R_k \to R $, and denote by $ u_k \equiv u_{R_k} $ the corresponding sequence of radial solutions to \eqref{loc-dir} which, as ensured by Claim 2, are also Sobolev minimizers of \eqref{def-IR} (with $ R \equiv R_k $) satisfying
\beq\label{MR to IR - k}
\int_0^{R_k} u_k^{q+1} \, \psi^{n-1} \, dr  = I_{R_k}^{2\frac{q+1}{q-1}} \qquad \forall k \in \mathbb{N} \, .
\eeq
Taking advantage of the same compactness argument exploited above, along with the continuity of $ R \mapsto I_R $, we have that $ \{ u_k \} $ converges (up to subsequences) strongly in $ H^1_0(B_{R+1}) $ to some $ \overline{u} \in H^1_{0,\mathrm{rad}}\!\left(B_R\right) $, which is therefore a positive Sobolev minimizer of \eqref{def-IR}; moreover, taking the limit of \eqref{MR to IR - k} as $k \to \infty$, we find that $ \overline{u} $ also satisfies  
$$
\int_0^{R} \overline{u}^{q+1} \, \psi^{n-1} \, dr  = I_{R}^{2\frac{q+1}{q-1}} \, .
$$
Hence, still by Claim 2, it follows that $ \overline{u}=u_R $, and the limit being independent of the chosen subsequence, the whole sequence $ \{ u_k \} $ converges to $u_R$. Finally we observe that, by elliptic regularity, the sequence $ \{ u_k\} $ is in fact locally uniformly H\"older continuous, whence 
$$
\lim_{k \to \infty} u_k(0) = u_R(0) \, ,
$$
namely the desired continuity of $A_R$. Indeed, the classical Calder\'on-Zygmund theory ensures that $ u_k \in W^{2,r}\!\left(B_{S}\right) $ for $r=(q+1)/q$ and all $S \in (0,R)$ (with uniform estimates in~$k$); then, after a finite number of iterations, we can infer that $ u_k \in C^{\beta}\!\left(B_{S}\right) $ for some $ \beta \in (0,1) $, still uniformly in $k$.
 
\medskip
\noindent \textbf{Claim 5:} \it The function $A(R)$ defined in \eqref{def-AR} is a bijection of $ (0,+\infty) $ onto itself.\rm 

\smallskip
\noindent Owing to Claim 4, it is enough to show that $A(R)$ is injective and that
\beq\label{lim-A}
\lim_{R \to 0} A(R) = +\infty \qquad \text{and} \qquad \lim_{R \to +\infty} A(R) = 0 \, .
\eeq
Injectivity is a direct consequence of the well-posedness (in particular, local uniqueness) of the ODE Cauchy problem \eqref{pb Cauchy}: if $ A(R_1) = A(R_2) $ for two different radii $0 < R_1 < R_2 $, then this would necessarily imply $ u_{R_1}(r) = u_{R_2}(r) $ for every $ r \in [0,R_1) $; however, since $ u_{R_1}(R_1) = 0 $ and $ u_{R_1}'(R_1)<0 $, the solution $u_{R_2}$ would necessarily change sign at $ r=R_1 $, which contradicts its positivity in $ [0,R_2) $. Hence, we have that $ A(R) $ is strictly monotone, so that the limits in \eqref{lim-A} do exist. As concerns the left one, it follows from the left identity in \eqref{lim-IR} along with \eqref{MR to IR} and the fact that $u_R$ is radially decreasing:
$$
A(R)^{q+1} \ge \frac{\int_0^R u_R^{q+1} \, \psi^{n-1} \, dr}{\int_0^R \psi^{n-1} \, dr}  = \frac{I_R^{2 \frac{q+1}{q-1}}}{ \int_0^R \psi^{n-1} \, dr } \rightarrow + \infty \qquad \text{as } R \to 0 \, .
$$
As a byproduct, we infer that $A(R)$ is decreasing: therefore, $ \{ u_R \} $ stays uniformly bounded as $R \to +\infty$. Note that, by combining \eqref{MR to IR}, \eqref{sobo-mass-2}, and the right identity in \eqref{lim-IR}, we find  
$$
\int_0^R \left| u_R' \right|^2 \psi^{n-1} \, dr = I_R^{2 \frac{q+1}{q-1}} \rightarrow 0 \qquad \text{as } R \to +\infty \, .
$$
In particular, along subsequences, it follows that $ u'_R \to 0 $ almost everywhere. Now, recalling the integrated form of the equation  \eqref{id-1}, we may deduce that
\beq\label{int-zero}
\lim_{R \to +\infty} \int_0^r u_R^q \, \psi^{n-1} \, ds = 0 \qquad \text{for a.e. } r>0 \, .
\eeq
Integrating further  \eqref{id-1} from $0$ to a fixed arbitrary $r>0$, we end up with
\beq\label{int-zero-2}
A(R) = u_R(r) +  \int_0^r \frac{\int_0^s u_R^q \, \psi^{n-1} \, dt}{\psi^{n-1}(s)} \, ds  \, .
\eeq
Using \eqref{int-zero} and the uniform boundedness of $ \{ u_R \} $, we can make sure that the integral in \eqref{int-zero-2} vanishes as $ R \to +\infty $. On the other hand, still \eqref{int-zero} guarantees that $ u_R \to 0 $ almost everywhere, up to subsequences; hence, by letting $ R \to +\infty $ in \eqref{int-zero-2} along any such subsequence, we obtain the right limit in \eqref{lim-A}.

\medskip
\noindent \textbf{Conclusion.} The global positive radial solution $u$ to \eqref{loc-dir} necessarily satisfies the local Cauchy problem \eqref{pb Cauchy} for some $a>0$. However, from Claim 5, we know that there exists some $R_\ast>0$ such that $ u_{R_\ast}(0)=a $, so that $ u=u_{R_*} $ in $ [0,R_*) $, but this is inconsistent with the positivity of $u$. A contradiction having been achieved, the thesis is proved.
\end{proof}

\section*{Acknowledgments} 
The authors thank the INdAM-GNAMPA research group. 
A. De Luca and N. Soave are partially supported by the INdAM-GNAMPA 2025 Project ``PDE ellittiche che degenerano su variet\`{a} di dimensione bassa e frontiere libere molto sottili'', CUP E5324001950001. N. Soave is partially supported by the PRIN Project no.~2022R537CS ``$NO^3$ -- Nodal optimization, NOnlinear elliptic equations, NOnlocal geometric problems, with a focus on regularity'', CUP D53D23005930006. M. Muratori is partially supported by the INdAM-GNAMPA 2025 Project ``Propriet\`a qualitative e regolarizzanti di equazioni ellittiche e paraboliche'', CUP E5324001950001, and by the PRIN Project no.~2022SLTHCE ``GAMPA -- Geometric-Analytic Methods for PDEs and Applications'', CUP E53D23005880006. Both of the PRIN grants are funded by the European Union - Next Generation EU within the PRIN 2022 program (D.D. 104 - 02 02 2022 Ministero dell'Universit\`a e della Ricerca, Italy).



\end{document}